\documentclass[11pt]{article}
\usepackage{amsmath,amssymb,amsthm, amscd, tikz, mathrsfs, titlesec}

\usepackage{graphicx}

\titleformat{\subsection}[runin]{\normalfont\bfseries}{\thesubsection.}{3pt}{}

\newtheorem{theorem}{Theorem}[section]
\newtheorem{lemma}[theorem]{Lemma}
\newtheorem{proposition}[theorem]{Proposition}
\newtheorem{corollary}[theorem]{Corollary}
\newtheorem{fact}[theorem]{Fact}

\newenvironment{definition}[1][Definition]{\begin{trivlist}
\item[\hskip \labelsep {\bfseries #1}]}{\end{trivlist}}
\newenvironment{example}[1][Example]{\begin{trivlist}
\item[\hskip \labelsep {\bfseries #1}]}{\end{trivlist}}
\newenvironment{remark}[1][Remark]{\begin{trivlist}
\item[\hskip \labelsep {\bfseries #1}]}{\end{trivlist}}

\usepackage{amssymb}

\author{Dale Winter}
 \title{Mixing of frame flow for rank one locally symmetric spaces and measure classification}

\begin{document} 

\maketitle

\begin{abstract} 


Let $G$ be a connected simple linear Lie group of rank one, and let $\Gamma <G$ be a discrete Zariski dense subgroup admitting a finite Bowen-Margulis-Sullivan measure $m^{\operatorname{BMS}}$. We show that the right translation action of the one dimensional diagonalizable subgroup is mixing on $(\Gamma \backslash G, m^{\operatorname{BMS}})$. Together with the work of Roblin, this proves ergodicity of the Burger-Roblin measure under the horospherical group $N$, establishes a classification theorem for $N$ invariant Radon measures on $\Gamma \backslash G$, and provides precise asymptotics for the Haar measure matrix coefficients. 
\end{abstract}

\section{Introduction} \label{sec1}

\subsection{Background and results.}
Mixing results have proven to be powerful tools in homogeneous dynamics. If $  G$ is a connected and simple Lie group, for instance, and if $ \Gamma <  G$ is a lattice, then the translation action of $ G$ on $ \Gamma \backslash G$ is known to be mixing for the Haar measure \cite{howe1979asymptotic}. Results like this, together with their effective refinements, have a huge range of applications including orbit counting asymptotics, shrinking target properties, and diophantine approximation \cite{duke_density_1993, eskin_mixing_1993, kleinbock_logarithm_1999, margulis2004some, oh_orbital_2010}.

The aim of this paper is to present and discuss a mixing result for thin subgroups of rank one matrix groups. Let $G$ be a connected simple linear group of real rank one, and let $\Gamma < G$ be discrete, torsion free, and non-elementary. For technical reasons we also assume that $G$ is not a finite cover of $\operatorname{PSL}_2(\mathbb{R})$. Associated to $G$ there is a rank one symmetric space $\tilde X$ with a negatively curved Riemannian metric $d$ induced by the Killing form. After re-scaling, we assume that the maximum sectional curvature of $d$ on $\tilde X$ is $-1$. The action of $G$ on $(\tilde X, d)$ is isometric, so extends to a (transitive) action of $G$ on the unit tangent bundle $\mathrm{T}^1\tilde X$; we identify $\tilde X = G/K$ and $\mathrm{T}^1 \tilde X = G/M$, where $K$ is the stabilizer of a point $o \in \tilde X$ and  $M<K$ is the stabilizer of a tangent vector $v_0 \in \mathrm{T}^1_o\tilde X$. Let $A=\lbrace a_t: t\in \mathbb{R}\rbrace< Z_G(M)$ be the one-parameter subgroup in the centralizer of $M$ such that the right action of $a_t$ on $G/M$ corresponds to geodesic flow on $\mathrm{T}^1\tilde X$ and such that $\mathrm{T}_eA $ is perpendicular to $\mathrm{T}_e K$.


Denote by  $m^{\operatorname{BMS}}$ the Bowen-Margulis-Sullivan measure on $\Gamma \backslash \mathrm{T}^1\tilde X = \Gamma \backslash G / M$. It is an $A$ invariant measure by construction (see subsection \ref{ss2.5}), and coincides with the measure of maximal entropy whenever $|m^{\operatorname{BMS}}| < \infty$ \cite{otal2004principe}. Babillot \cite{babillot_mixing_2002} showed  that the $A$ action on $(\Gamma \backslash G/M, m^{\operatorname{BMS}})$ is mixing whenever $|m^{\operatorname{BMS}}|<\infty$; see also the work of Kim \cite{kim_length_2006}, who proved that Babillot's assumption on non-arithmeticity of the length spectrum holds in our setting.


Our aim in the current work is to lift mixing of the geodesic flow from the unit tangent bundle $\Gamma \backslash G/M$ up to the ``frame bundle" $\Gamma\backslash G$ when $\Gamma < G$ is Zariski dense. More precisely, we lift the BMS measure on $\Gamma \backslash G/M$ to give an $MA$ invariant measure on $\Gamma\backslash G$ which, abusing notation, we will also denote by  $m^{\operatorname{BMS}}$. The standing assumptions for the introduction are that 
\[\mbox{ $\Gamma$ is Zariski dense and $|m^{\operatorname{BMS}}| < \infty$}.\]
Here is our main theorem; it is, in some sense, a generalization of Howe-Moore's result on the decay of matrix coefficients \cite{howe1979asymptotic} from lattices to general discrete groups with a finite BMS measure. 
\begin{theorem}[Mixing of frame flow] \label{target}
The right translation action of $A$ on $(\Gamma \backslash G, m^{\operatorname{BMS}})$ is mixing: for any $\psi_1, \psi_2 \in C_c(\Gamma \backslash G)$ we have
\[ \lim_{t\rightarrow\infty} \int_{\Gamma \backslash G} \psi_1(ga_t) \psi_2(g) dm^{\operatorname{BMS}}(g) =  \frac{m^{\operatorname{BMS}}(\psi_1)m^{\operatorname{BMS}}(\psi_2)}{|m^{\operatorname{BMS}}|}. \]
\end{theorem}
Results like Theorem \ref{target} are useful tools for orbit counting problems and equidistribution of translates of symmetric subgroups; see \cite{roblin_ergodicite_2003, kontorovich_apollonian_2011, oh_equidistribution_2013, mohammadi_matrix_2012} and also the survey paper \cite{oh_harmonic_2012}. Many papers in this area cite the work of Flaminio-Spatzier \cite{flaminio_geometrically_1990}, where Theorem \ref{target} was claimed under the additional assumptions that $\tilde X$ has constant curvature and that $\Gamma$ is geometrically finite. Unfortunately their argument appears to have a gap. For $\operatorname{dim}(\tilde X) \geq 4$ this gap is easily fixed using a commutator argument; the case $\operatorname{dim}(\tilde X) = 3$ is subtler, and led to the approach of Section \ref{sec4}. 

In fact, we will show in Section \ref{sec8} that $(\Gamma \backslash G, m^{\operatorname{BMS}}, A)$ is a $K$-system, and is Bernoulli whenever the geodesic flow ${(\Gamma \backslash G / M, m^{\operatorname{BMS}}, A)}$ is Bernoulli.

The finiteness assumption on $m^{\operatorname{BMS}}$ is fulfilled in many interesting cases. Recall that a group $\Gamma < G$ is geometrically finite if the unit neighborhood of the convex core of $\Gamma$ has finite Riemannian volume; this gives a natural generalization of the class of lattices in $G$ \cite{bowditch1995geometrical}.  We note that geometrical finiteness of $\Gamma$ implies finiteness of $m^{\operatorname{BMS}}$ \cite{sullivan_entropy_1984, corlette1999limit}. See also the work of Peign\'e for examples of geometrically infinite groups with finite BMS measure \cite{peigne2003patterson}. 

The proof of Theorem \ref{target} is actually rather general; one can replace the Bowen-Margulis-Sullivan measure by the $M$ invariant lift of any finite and $A$ ergodic quasi-product measure on $\Gamma \backslash G / M$. 


If $G$ is the identity component $\operatorname{SO}(n, 1)_+$ of the indefinite orthogonal group and $\tilde X = \mathbb{H}_\mathbb{R}^n$ is a real hyperbolic space (respectively if $G = \operatorname{SU}(n, 1)$ and  $\tilde X = \mathbb{H}_\mathbb{C}^n$ is a complex hyperbolic space), then $G$ acts simply transitively on the oriented orthonormal $n$-frame bundle over $\mathbb{H}_\mathbb{R}^n$ (respectively on the unitary $n$-frame bundle over $\mathbb{H}_\mathbb{C}^n$). In this case $\Gamma \backslash G$ is identified with the frame bundle over $\Gamma \backslash \tilde X$, and the $A$ action on $\Gamma \backslash G$ corresponds to frame flow. It is for this reason that we refer to the main theorem as ``mixing of frame flow". The geometric interpretation is not perfect, however: when $\tilde X$ is a quaternionic hyperbolic space, the $G$ action is no longer free on the (quaternionic) $n$-frame bundle; when $\tilde X$ is the Cayley hyperbolic plane, it's unclear even how we would define the (octonionic) frame bundle.

We denote by $N$ the expanding horospherical group. Using the mixing of the geodesic flow, Roblin proved that there is a unique $NM$ invariant and ergodic measure on $\Gamma \backslash G$, now called the Burger-Roblin measure, which is not supported on a  closed $NM$ orbit \cite{roblin_ergodicite_2003}. We denote the Burger-Roblin measure by $m^{\operatorname{BR}}$ (see subsection \ref{ss5.3}); it is known to be infinite unless $\Gamma$ is a lattice in $G$. 
\begin{theorem} \label{target2}
The right translation action of $N$ on $(\Gamma \backslash G, m^{\operatorname{BR}})$ is ergodic. 
\end{theorem}	

This result, together with Roblin's work, allows a classification of all $N$ invariant Radon measures on $\Gamma \backslash G$ for $\Gamma$ geometrically finite. Before we state our classification theorem, we observe that there are other $N$ invariant measures on $\Gamma\backslash G$ coming from closed $NM$ orbits. For any $g \in G$ such that  $\Gamma gNM \subset G$ is closed, we define the group 
\[ M_0(g) = \lbrace m \in M : mh \in g^{-1}\Gamma g \mbox{ for some } h \in N \rbrace < M. \]
This is known to be a virtually abelian subgroup of $M$ (see \cite[Theorem 8.24]{raghunathan_discrete_1972}). Let $L(g) = \Gamma \backslash \Gamma gN\overline{M_0(g)}\subset \Gamma \backslash G$; it is a minimal closed $N$ invariant subset and carries a unique (up to proportionality) Radon measure invariant for $N\overline{M_0(g)}$ (see Section \ref{sec7}).
\begin{theorem} \label{measureclassificati0on} Suppose that $\Gamma < G$ is geometrically finite.  If $\nu$ is an $N\negmedspace$ invariant and ergodic Radon measure on $\Gamma \backslash G$, then, up to proportionality, we have either
\begin{itemize}
\item{ $\nu = m^{\operatorname{BR}}$ or }
\item{ $\nu$ is supported on some $L(g)$ and is the unique $N\overline{M_0(g)}$ invariant measure.} 
\end{itemize}
\end{theorem}
We can also use mixing of $(\Gamma \backslash G, m^{\operatorname{BMS}}, A)$ to compute the precise asymptotics for the decay of matrix coefficients with respect to the Haar measure. In the case when $\Gamma$ is a lattice in $G$, the following theorem simply reduces to the Howe-Moore theorem on the decay matrix coefficients. In general it follows from Theorem \ref{target} using the arguments of \cite[Chapter III]{roblin_ergodicite_2003}. 
\begin{theorem} \label{T1.3}
Denote by $m^{\operatorname{BR}}_*$ the \textup{BR} measure for the contracting horospherical subgroup, and by $m^{\operatorname{Haar}}$ the Haar measure on $\Gamma \backslash G$. For any ${\psi_1, \psi_2 \in C_c(\Gamma \backslash G)}$, we have 
\[  \lim_{t\rightarrow +\infty}  e^{(D - \delta) t} \int_{\Gamma \backslash G} \psi_1(g a_t) \psi_2(g) dm^{\operatorname{Haar}}(g) = \frac{m^{\operatorname{BR}}(\psi_1) m^{\operatorname{BR}}_*(\psi_2) }{|m^{\operatorname{BMS}}|}, \]
where $\delta$ is the critical exponent of the group $\Gamma$, and $D$ is the volume entropy of the symmetric space associated to $G$ (see subsection \ref{ss5.3}).
\end{theorem}

Other applications of this result seem possible. For example, one should be able to use Theorem \ref{target} to obtain orbit counting asymptotics in the style of \cite{mohammadi_matrix_2012, oh_equidistribution_2013} for the action of $\Gamma$ on quotient spaces of $G$. Another natural question is to what extent the mixing and equidistribution results we state here can be made effective. For real hyperbolic spaces, this question has already been studied when the critical exponent is large \cite[Theorems 1.4, 1.6, 1.7]{mohammadi_matrix_2012}, and it is expected that those results can be generalized to other rank one spaces. 

\subsection{Outline of the proof of Theorem \ref{target}.}

The proof of Theorem \ref{target} breaks into two steps: a geometric argument, and a measure theoretic argument. For the geometric argument, we will define and study the transitivity group $\mathcal{H}_\Gamma(g)$ associated to each element $g$ in the support of the BMS measure. This is inspired by the ideas of Brin for compact manifolds \cite{brin_ergodic_1979}  and Flaminio-Spatzier for geometrically finite manifolds of constant curvature \cite{flaminio_geometrically_1990}. Traditionally, the transitivity group is defined as a subgroup of the structure group $M$. However, it is important in our setting to define the transitivity group as a subgroup $\mathcal{H}_\Gamma(g) < MA$. The first key step in the proof of Theorem \ref{target} is to show that the transitivity groups are dense in $MA$. We use algebraic properties of the Bruhat decomposition to see that $\overline{\mathcal{H}_\Gamma(g)}$ contains the commutator $[MA, MA]$. By working with the Lie algebra $\mathrm{T}_eG$, we then define a smooth function from the limit set $\Lambda(\Gamma)$ into the transitivity group; this allows us to convert properties of the transitivity group (such as not being dense) into constraints on the limit set (such as being contained in a smooth submanifold of the boundary). Finally we use expansion invariance of $\Lambda(\Gamma)$ in conjunction with those smooth constraints to force density of $\mathcal{H}_\Gamma(g)$. 

The next step is a measure theoretic argument, using density of the transitivity group to prove mixing of the frame flow; this is a variation on work of Babillot \cite{babillot_mixing_2002}. It is sufficient to show that if $\psi \in C_c(\Gamma \backslash G)$, then any $L^2(\Gamma \backslash G, m^{\operatorname{BMS}})$ weak limit of the sequence $\psi^{a_t}(g) := \psi(ga_t)$ is constant. It is known that any weak limit $\phi$ of $\psi^{a_t}$ has measure theoretic invariance under both the expanding and contracting horospherical groups. Using smoothing and the product structure of the BMS measure, we deduce that $\phi$ is invariant under the transitivity group. Density of the transitivity group then implies that $\phi$ is $MA$ invariant and so, using ergodicity of $(\Gamma \backslash G / M, m^{\operatorname{BMS}}, A)$, that $\phi$ is indeed constant. 


\subsection{Organization of the paper.}

We begin in Section \ref{sec2} by fixing notation and recalling some necessary background material; of particular importance are the Bruhat decomposition for $G$, the construction of generalized BMS measures, and certain properties of geometrically finite groups.

	The argument begins in earnest in Section \ref{sec3}, where we define the transitivity groups and establish their basic properties. The main aim here is to show that the transitivity groups are indeed dense in $MA$.

	In Section \ref{sec5} we start our foray into measure theory; we begin with some technicalities, showing that measure theoretic invariance under the horospherical groups implies invariance under the transitivity group. This allows us to complete the proof of Theorem \ref{target} in Section \ref{sec6}, where we study weak limits of continuous functions under frame flow and show that they must be constant. 
	
	Section \ref{sec7} is concerned with applications: it begins with results on the equidistribution of horospheres and the decay of matrix coefficients, before using those results to establish the measure classification for horospherical groups.  The paper closes with Section \ref{sec8}, in which we refine the arguments of Section \ref{sec6} to give Bernoulliness or the Kolmogorov property for frame flow.

\subsection*{Acknowledgements:} I would like to thank my advisor, Hee Oh, for suggesting the problem and for her advice and many helpful comments throughout the project. I would also like to thank Ralf Spatzier for a very useful discussion of the constant curvature case which led to the formulation of Corollary \ref{inasmoothsubmanifold}, and Frederic Paulin, for pointing out that the mixing result might generalize to Gibbs and quasi-product measures.

\section{Preliminaries} \label{sec2} The first aim of this section is to establish notation. We will then review some background material that will be needed later on.

\subsection{Notation} We begin by fixing some notation, which will be used without comment throughout the rest of the paper. Let $G$ be a connected simple linear group of real rank one. As in the introduction we exclude the case that $G$ is a finite cover of $\operatorname{PSL}_2(\mathbb{R})$ unless explicitly stated otherwise. The classification theorem tells us that $G$ is locally isomorphic to one of the following groups; $\operatorname{SU}(n, 1)$, $\operatorname{Sp}(n, 1)$, $\operatorname{F}^{-20}_4$, or the identity component $\operatorname{SO}(n, 1)_+$ with $n \geq 3$. Let $K<G$ be a maximal compact subgroup, and let $\theta: G \rightarrow G$ be the associated Cartan involution. The Killing form gives a left $G$ invariant Riemannian metric $d$ on $\tilde X := G/K$, which we normalize so that the maximum sectional curvature of $(\tilde X, d)$ is minus one. We will write $\partial \tilde X$ for the geometric boundary of $\tilde X$ (see subsection \ref{properties of U and N}).

The letter $\Gamma$ will always denote a discrete, torsion free, and non-elementary subgroup of $G$. We will write $\Lambda(\Gamma)$ for the limit set of $\Gamma$ in $\partial \tilde X$.

Let $o = [K]$ be the natural base point of $\tilde X$, and fix a unit tangent vector $v_0 \in \mathrm{T}^1_{o}\tilde X$. We will write $\mathfrak{g}$ for the real Lie algebra $\mathrm{T}_eG$. As always, $\theta$ induces an orthogonal splitting of $\mathfrak{g}$ into a $(+1)$ eigenspace $\mathfrak k = \mathrm{T}_eK$ and a $(-1)$ eigenspace $\mathfrak p$;
\[ \mathfrak{g} = \mathfrak k \oplus \mathfrak{p}. \]
Write $\pi : G \rightarrow \tilde X$ for the natural projection and observe that $\mathfrak{k}$ is exactly the kernel of $d\pi: \mathfrak{g} \rightarrow \mathrm{T}_e(\tilde X)$. It follows that we can lift $ v_0$ to a unique unit vector $\hat v_0 \in \mathfrak{p}$. The one parameter subgroup
\[ A =\lbrace a_t :=  \mbox{exp}(t \hat v_0 ): t\in \mathbb{R} \rbrace < G \]
projects to a geodesic $a_to$ in $(\tilde X, d)$. We will write $M$ for the stabilizer $\mbox{stab}_{K}(v_0)$ of $v_0$ in $K$. We denote by $N^+$ (respectively $N^-$) the expanding (contracting) horospherical groups:
\[ N^+ = \lbrace h \in G: a_{-t} h a_t \rightarrow e \mbox{ as } t\rightarrow -\infty \rbrace; \]
\[ N^- = \lbrace n \in G: a_{-t} n a_t \rightarrow e \mbox{ as } t\rightarrow +\infty\rbrace. \]

\subsection{Properties of $M$.}The first remark is that we can characterize $M$ without reference to the symmetric space $\tilde X$. 
\begin{lemma}
The centralizer $Z_K(A)$ is equal to $M$. 
\end{lemma}
\begin{proof}
Note first that if $m \in Z_K(A)$, then $ma_t K= a_tmK = a_tK$. Taking derivatives at time zero we see that $m \in \mbox{stab}_K( v_0)$ as expected.
On the other hand, if $m \in \mbox{stab}_K( v_0)$, then
\begin{eqnarray*}  v_0 &=& m_* v_0 \\
&=& d\pi \left(\mathrm{Ad}_m \hat  v_0\right).\end{eqnarray*}
It follows that $\hat v_0 - \mathrm{Ad}_m \hat v_0 \in \mathfrak{k}$, the kernel of $d\pi$. But both $ \hat v_0$ and $\mathrm{Ad}_m \hat v_0$ are in $\mathfrak{p}$. Thus 
\[ \hat v_0 - \mathrm{Ad}_m \hat v_0 \in \mathfrak{k} \cap \mathfrak{p}  = \lbrace 0 \rbrace \]
and $m \in Z_K(A)$ as required. 
\end{proof}
The following fact will be needed later. 
\begin{fact}
The dimension of the center of $M$ is at most one. 
\end{fact}
\noindent To see this we write down the group $M$ in every case that interests us (see \cite{kim_counting_2011, kim2003geometry}). If $G = \operatorname{SO}(n, 1)_+$, then $M$ is $\operatorname{SO}(n-1)$. The other cases are:

\begin{itemize} 
\item { if $G = \operatorname{Sp}(n, 1)$, then $M = \operatorname{SO} (3) \times \operatorname{Sp}(n-1)$;}
\item { if $G = \operatorname{SU}(n, 1)$, then 
\begin{eqnarray*}  M \negthickspace \negthickspace \negthickspace &=& \negthickspace \operatorname{S}(\operatorname{U}(n-1) \times \operatorname{U}(1))\\\negthickspace
&:=& \negthickspace  \left\lbrace S =  \left(   \begin{array}{ccc} R &0 &0\\ 0 &\alpha &0 \\ 0 & 0 & \alpha \end{array} \right): R \in \operatorname{U}(n-1) , \alpha \in \mathbb{C}, \mbox{ and } \det S = 1 \right\rbrace;  \end{eqnarray*}  
}  
\item{ and if $G = \mathrm{F}_4^{-20}$,  then $M = \operatorname{Spin}(7)$.}
\end{itemize} 
Of these groups, the only ones with non-discrete center are $\operatorname{S}(\operatorname{U}(n-1) \times \operatorname{U}(1))$ and $\operatorname{SO}(2)$, where the centers are one dimensional.

\subsection{Algebraic properties of the Bruhat decomposition.} 
For the rest of this section, a blackboard bold character (such as $\mathbb{G}$) will always refer to an $\mathbb{R}$-algebraic subgroup of $\mathrm{GL}_n(\mathbb{C})$. We will denote the real points of such a group by $\mathbb{G(R)}$. It will be necessary to use both the Zariski topology and the Euclidean topology and to switch regularly between the two. A subscript ``0'' will denote the identity component of a group in the Zariski topology, so for instance $\mathbb{G}_0$ will be the Zariski identity component of $\mathbb{G}$. A subscript ``+'' will mean the Euclidean identity component, e.g. $\mathbb{G(R)}_+$ is the Euclidean identity component in the real points of $\mathbb{G}$. The words open, closed, dense, compact, and so forth used without qualification will refer to the Euclidean topology. When referring to the Zariski topology we will be explicit.

All of the groups we've discussed so far are actually algebraic. More precisely, there is a Zariski connected $\mathbb{R}$-algebraic subgroup $\mathbb{G} < \mathrm{GL}_n(\mathbb{C})$ with $\mathbb{R}$-rank one and $G = \mathbb{G(R)}_+$ is the identity component in the group of real points. We denote by $\mathbb{K}$ the Zariski closure of $K$ in $\mathbb{G}$. Algebraicity of compact real matrix groups implies that $\mathbb{K}(\mathbb{R}) = K$. We choose a maximal $\mathbb{R}$-split torus $\mathbb{S} < \mathbb{G}$ such that $\mathbb{S(R)}_+ = A$. Let $\mathbb{M} = Z_\mathbb{K}(\mathbb{S})$ and again note that  $M = \mathbb{M(R)} < K$. Finally let $\mathbb{N}^-$ be the negative root group of $\mathbb{S}$ and $\mathbb{N}^+$ the opposite horospherical group. Once again we have $N^- = \mathbb{N^-(R)}$ and $N^+ = \mathbb{N^+(R)}$. 

Let $\mathbb{P}$ be the minimal $\mathbb{R}$-parabolic subgroup of $\mathbb{G}$ containing $\mathbb{S}$ and $\mathbb{N^-}$, and let $\mathbb{W} = \mathbb{N^+.P} \subset \mathbb{G}$. The Bruhat decomposition tells us that $\mathbb{W}$ is Zariski open in $\mathbb{G}$ and that the product map
\[ \mathbb{N^+ \times P} \rightarrow \mathbb{W}\]
is an $\mathbb{R}$-isomorphism of varieties. Now let $\mathbb{L} = Z_\mathbb{G}(\mathbb{S})$ be the centralizer of $\mathbb{S}$ in $\mathbb{G}$. This is a Levi factor for $\mathbb{P}$, so the product map
\begin{equation} \mathbb{W} = \mathbb{N^+ \times L \times N^-} \label{eqn1} \end{equation}
is also an $\mathbb{R}$-isomorphism of varieties. We begin by identifying the Levi factor.

\begin{lemma} \label{L2.3}
We have $\mathbb{L}(\mathbb{R})_+ = M_+A$ and $\mathbb{L} = \mathbb{M}_0.\mathbb{S}$.
\end{lemma}
\begin{proof}
Clearly $M_+A \subset \mathbb{L}(\mathbb{R})_+$. On the other hand 
\begin{eqnarray*} \mathrm{T}_e\mathbb{L(R)}_+ &=& Z_{\mathfrak{g}}(\hat  v_0)\\
&=& \mathrm{T}_e(M_+A), \end{eqnarray*}
and the first claim follows. For the second claim we note that  $\mathbb{M}_0.\mathbb{S} \subset \mathbb{L}$. On the other hand $\mathbb{M}_0.\mathbb{S}$ is Zariski closed, and $\mathbb{L(R)}_+ \subset \mathbb{L}$ is Zariski dense, so $\mathbb{L} \subset \mathbb{M}_0.\mathbb{S}$ as expected. 
\end{proof}

\begin{lemma} \label{L2.4}
The group $M$ is connected.
\end{lemma}
\begin{remark} This is the stage at which finite covers of $\operatorname{PSL}_2(\mathbb{R})$ must be handled separately.
\end{remark}
\begin{proof}
By the relative Bruhat decomposition, the complement
\[ \mathbb{G}(\mathbb{R}) \setminus \mathbb{W}(\mathbb{R})  \simeq \mathbb{P}(\mathbb{R})\]
has codimension at least two in $\mathbb{G}(\mathbb{R})$. It follows that 
\[ \mathbb{W}(\mathbb{R})_+ = G \cap \mathbb{W}(\mathbb{R}).\]
As a consequence of this we have
\begin{eqnarray*} M &\subset& G \cap \mathbb{L}(\mathbb{R})\\
& \subset & \mathbb{W}(\mathbb{R})_+\\
& =&  N^+ \times \mathbb{L}(\mathbb{R})_+ \times N^-. \end{eqnarray*}
Thus $M \subset \mathbb{L}(\mathbb{R})_+$ by uniqueness properties of the decomposition (\ref{eqn1}), and $M \subset M_+$ by Lemma \ref{L2.3}.
\end{proof}
\begin{corollary}
The product map
\[ N^+ \times M \times A \times N^- \rightarrow \mathbb{W}(\mathbb{R}) \cap G\]
is a diffeomorphism.
\end{corollary}
\begin{proof}
Lemmas \ref{L2.3}  and \ref{L2.4} imply that $M \times A \rightarrow \mathbb{L}(\mathbb{R})_+$ is a diffeomorphism. The claim now follows by the decomposition (\ref{eqn1}). 
\end{proof}
We're now in a position to state one of the two technical facts we need from this section. Let $W = \mathbb{W}(\mathbb{R})_+$. We write
\[ \zeta = \zeta_{N^+} \times \zeta_M \times \zeta_A \times \zeta_{N^-}: W \rightarrow N^+ \times M\times A\times N^-. \]
for the inverse to the multiplication map.
 
\begin{proposition} \label{dense in M}
Suppose that $\Gamma < G$ is Zariski dense. Then $\zeta_M(\Gamma \cap W) $ is Zariski dense in $M$, and generates a dense subgroup $\langle \zeta_M(\Gamma \cap W)\rangle < M$. \end{proposition}
Before proving this we need two lemmas. 
\begin{lemma}
The intersection  $\mathbb{M}_0 \cap \mathbb{S}$ is finite. 
\end{lemma}
\begin{proof}
Write $\mathbb{H}$ for the intersection $\mathbb{M}_0 \cap \mathbb{S}$. We note that $\mathbb{H(R)} \subset \mathbb{M}_0(\mathbb{R})$ is compact. But every element of $\mathbb{S(R)}$ other than $\pm1$ generates a non-compact subgroup of $\mathbb{G(R)}_+$. It follows that $\mathbb{H}_0(\mathbb{R}) \subset \mathbb{H(R)}$ has at most two elements. The real points in $\mathbb{H}_0$ are Zariski dense, so $\mathbb{H}_0$ is finite, and therefore trivial.  Finally note that $\mathbb{H}_0$ has finite index in $\mathbb{H}$ itself, so we conclude that $\mathbb{M}_0 \cap \mathbb{S} = \mathbb{H}$ is also finite. 
\end{proof}
\begin{lemma} \label{L2.7} Let 
\[p:  \mathbb{M}_0 \times \mathbb{S} \rightarrow \mathbb{L}\] 
be the product map and suppose that $\Gamma \subset \mathbb{L}$ is Zariski dense. Then the pre-image $p^{-1}(\Gamma) \subset \mathbb{M}_0 \times \mathbb{S}$ is Zariski dense. 
\end{lemma}

\begin{proof} Suppose not, that is suppose that the Zariski closure $V = \overline{ p^{-1}( \Gamma)}$ is a proper subvariety of $\mathbb{M}_0 \times \mathbb{S}$. Then the image $p(V)$ cannot contain any Zariski open subsets of $\mathbb{\mathbb{L}}$ for dimensional reasons. On the other hand, $p(V)$ is constructible. It follows that $p(V)$ is contained in a finite union of proper subvarieties of $\mathbb{L}$. This contradicts Zariski density of ${\Gamma}$. \end{proof}

\begin{proof}[Proof of Proposition \ref{dense in M}] Suppose that $\Gamma \subset G$ is Zariski dense. Then $\Gamma \cap W = \Gamma \cap \mathbb{W}$ is Zariski dense in $\mathbb{W}$. If we write $\zeta_\mathbb{L}$ for the $\mathbb{L}$ component of the isomorphism $ \mathbb{W} \rightarrow \mathbb{N^+ \times L \times N^-}$ above, then we conclude that $\zeta_\mathbb{L}(\Gamma \cap W ) \subset \mathbb{L}$ is Zariski dense. By Lemma \ref{L2.7}, it follows that 
\[ \tilde \Gamma :=p^{-1}(\zeta_\mathbb{L}(\Gamma \cap W )) \subset \mathbb{M}_0 \times \mathbb{S}\]
is Zariski dense, where $p: \mathbb{M}_0\times \mathbb{S} \rightarrow \mathbb{L}$ is the product map. But $\tilde \Gamma$ is very nearly the same as $\zeta_{M} \times \zeta_{A} (\Gamma \cap W )$; in fact
\[ \tilde \Gamma = \bigcup_{g \in \mathbb{M}_0\cap \mathbb{S}} \lbrace ( g^{-1}\zeta_{M}(\gamma  ),  g\zeta_{A} (\gamma )): \gamma \in \Gamma \cap W \rbrace. \]
This is a union of finitely many sets and it follows (since $\mathbb{M}_0\times \mathbb{S}$ is irreducible) that each set
\[ \lbrace ( g^{-1}\zeta_{M}(\gamma  ),  g\zeta_{A} (\gamma )): \gamma \in \Gamma \cap W \rbrace\]
is Zariski dense.  In particular, the case $g = e$ implies that $\zeta_M(\Gamma \cap W ) \subset \mathbb{M}_0$ is Zariski dense as required. 
The second statement is quicker to prove. The subgroup $H<M$ generated by $\zeta_M(\Gamma \cap W )$ is Zariski dense in $M$. The Euclidean closure $\bar{H}$ is therefore both Zariski closed (by algebraicity of compact matrix groups) and Zariski dense, so $\bar{H} = M$. 
\end{proof}

\subsection{Properties of $N^-$ and $N^+$ and the boundary at infinity.} \label{properties of U and N} The geometric boundary $\partial \tilde X$ of  $\tilde X$ may be identified with the set of unit speed  geodesic rays $x(t)$ in $\tilde X$ under the equivalence relation $x_1(t) \sim x_2(t)$ if $d(x_1(t), x_2(t))$ is bounded. We recall that the isometric action of $G$ extends continuously to the boundary. For any tangent vector $v \in \mathrm{T}^1\tilde X$ we write $ v^+$ for the forward end point of the associated geodesic $v(t)$ in $\partial  \tilde X$, and  $v^-$ for the backward end point. Of particular importance are the forward and backward end points, $v_0^+$ and $v_0^-$,  of our fixed tangent vector $v_0$. It will also be convenient to talk about the forward and backward end points of group elements, by which we mean
\[ g^+ := (g v_0)^+ \mbox{  and } g^- := (g v_0)^-.\]
Note that $g^+ = g(v_0^+)$ and that $g^- = g(v_0^-)$. It can be seen that \\ ${P :=\mathbb{P(R)} =  MAN^-}$ is the stabilizer subgroup of $v_0^+$. On the other hand the map
\[ N^- \rightarrow \partial \tilde X, n \mapsto n^- \]
is a diffeomorphism onto $\partial \tilde X \setminus \lbrace v_0^+ \rbrace$, so $N^-$ doesn't stabilize $v_0^-$. We denote the inverse map by
\[ \partial \tilde X \setminus \lbrace v_0^+ \rbrace \rightarrow N^-, \xi \mapsto n(\xi).\]
Similarly $N^+MA$ is the stabilizer of $v_0^-$, while
\[ N^+ \rightarrow \partial \tilde X, h \mapsto h^+ \]
is a diffeomorphism onto $\partial \tilde X \setminus \lbrace v_0^- \rbrace$, whose inverse we denote by $\xi \mapsto h(\xi)$.

\begin{proposition}\label{dense in U and N}
Suppose that $\Gamma < G$ is Zariski dense and that $\xi \in \partial \tilde X$. Then
\begin{itemize}  
\item{ the set $\lbrace h(\gamma\xi): \gamma \in \Gamma, \gamma \xi \neq v_0^- \rbrace$ is Zariski dense in $N^+$, and }
\item{ the set $\lbrace n(\gamma\xi): \gamma\in \Gamma, \gamma\xi \neq v_0^+ \rbrace$ is Zariski dense in $N^-$.}
\end{itemize} 
\end{proposition}

\begin{proof}We identify $G / P$ with $\partial \tilde X $ via the orbit map
\[ f: gP \mapsto g(v_0^+) = g^+.\]
Note that $f$ is equivariant for the left $G$ action. Under this identification $W/P$ corresponds exactly to $\partial \tilde X \setminus \lbrace v_0^-\rbrace$. Let 
\[ \psi: \mathbb{W/P} = \mathbb{N}^+ \mathbb{\times P/P} \rightarrow \mathbb{N}^+ \]
be the natural projection. Note that 
\[ \psi ( f^{-1}(\zeta)) = h(\zeta)\]
whenever $\zeta \in \partial \tilde X \setminus \lbrace v_0^- \rbrace $. Now let $\xi \in \partial \tilde X$ and suppose that $\Gamma \subset G$ is Zariski dense. Then $\Gamma f^{-1}(\xi)$ is Zariski dense in $ \mathbb{G}/\mathbb{P}$. It follows that 
\begin{eqnarray*} & &(\Gamma f^{-1}(\xi) \cap \mathbb{W/P})\subset \mathbb{W/P} \mbox{ is Zariski dense } \\
&\Rightarrow&  (f^{-1}( \Gamma \xi \setminus \lbrace v_0^- \rbrace)) \subset \mathbb{W/P} \mbox{ is Zariski dense } \\
&\Rightarrow&  \psi \circ f^{-1}( \Gamma \xi \setminus \lbrace v_0^- \rbrace) \subset \mathbb{N}^+ \mbox{ is Zariski dense } \\
&\Rightarrow&  \lbrace h(\gamma\xi): \gamma\in\Gamma, \gamma\xi \neq v_0^- \rbrace  \subset \mathbb{N}^+ \mbox{ is Zariski dense }\end{eqnarray*}
as required. The other statement is similar. 
\end{proof}


\subsection{Bowen-Margulis-Sullivan measures.} \label{ss2.5}
Next we recall the construction of the BMS measures. This is essentially a synopsis of the necessary material from Chapter I of \cite{roblin_ergodicite_2003}.

As in Section 2.1, we consider the unit tangent bundle $G/M = \mathrm{T}^1\tilde X$. We will write 
\[ \partial ^2 \tilde X := (\partial \tilde X)^2 \setminus \lbrace \mbox{diagonal} \rbrace. \]
We denote the Busemann function by
\begin{eqnarray*} \beta: (\partial \tilde X) \times \tilde X \times \tilde X&\rightarrow& \mathbb{R}  \\
 \beta_\xi(x_1, x_2) &=& \lim_{t\rightarrow + \infty} [d (x_1, \xi(t)) - d(x_2, \xi(t))],\end{eqnarray*}
where $\xi(t)$ is any geodesic whose forward end point is $\xi \in \partial \tilde X$. Using the projection maps  $\mathrm{T}^1\tilde X \rightarrow \tilde X$ and $G \rightarrow \tilde X$  we may also evaluate the Busemann function on tangent vectors or on group elements.

A conformal density of dimension $\delta$ on $\tilde X$ is a family of measures\\ $ {\mu = \lbrace \mu_x : x \in \tilde X\rbrace} $ on $\partial \tilde X$ which are mutually absolutely continuous and satisfy 
\[ \frac{d\mu_{ x_1}}{d\mu_{x_2}} (\xi) = e^{\delta \beta_\xi(x_2, x_1)} \]
for all pairs $x_i \in \tilde X$. The conformal density is said to be invariant for a subgroup $\Gamma < G$ if
\[ \gamma_*\mu_x  = \mu_{\gamma x}\]
for all $\gamma\in\Gamma$ and all $x \in \tilde X$. 

Conformal densities allow us to construct measures on $\mathrm{T}^1\tilde X$. Our choice of base point $o \in \tilde X$ gives a diffeomorphism
\begin{eqnarray*} \mathrm{T}^1\tilde X &\rightarrow& \partial^2 \tilde X \times \mathbb{R}\\
 v &\mapsto& (v^+, v^-, \beta_{v^-}(v, x_0) ). \end{eqnarray*}
For $\Gamma$ invariant conformal densities $\nu, \mu$ of dimensions $\delta_\nu, \delta_\mu$ we can therefore define a measure on $\mathrm{T}^1\tilde X$ by
 \[ d\tilde m^{\nu, \mu}(v) =e^{\delta_\nu\beta_{v^-}( o, v) + \delta_\mu\beta_{v^+}(o, v) } d\nu_{o}(v^-)d\mu_{o}(v^+) ds.\]
Direct calculation shows that this measure is independent of the choice of base point and left $\Gamma$ invariant. It is invariant for geodesic flow whenever $\delta_\mu = \delta_\nu$. We now lift $\tilde m^{\nu, \mu}$ from $\mathrm{T}^1\tilde X$ to give an $M$ invariant measure on $G$, which we will also denote by $\tilde m^{\nu, \mu}$. Since $\tilde m^{\nu, \mu}$ is $\Gamma$ invariant, it then descends to give a measure $m^{\nu, \mu}$ on $\Gamma \setminus G$, which we will call the generalized BMS measure associated to the pair $\nu, \mu$.

Certain conformal densities lead to particularly interesting measures. One of these is the Patterson-Sullivan density; if $\Gamma<G$ is of divergence type (for example, if $\Gamma$ is geometrically finite), with critical exponent $\delta$, then there is a unique $\Gamma$ invariant conformal density of dimension $\delta$ supported on $\Lambda(\Gamma)$, called the Patterson-Sullivan density $\sigma$. When we talk about the BMS measure without qualification we mean $m^{\operatorname{BMS}} := m^{\sigma, \sigma}$; it is an $MA$ invariant measure on $\Gamma \backslash G$, and our principle object of study. The support of the BMS measure will play a significant role in what follows, so observe that
\[ \mathrm{supp}(m^{\operatorname{BMS}}) = \lbrace g \in G : g^ + \in \Lambda (\Gamma) \mbox{ and } g^- \in \Lambda(\Gamma) \rbrace .\]

\subsection{} \label{ss5.3} There is also a unique $G$ invariant conformal density whose dimension is given by the volume entropy $D= D(\tilde X)$ (see the next paragraph); we will denote it by $\lambda$. We can express $\lambda_o$ explicitly as the unique $K$ invariant probability measure on $\partial \tilde X$. This allows us to construct the Burger-Roblin measure $m^{\operatorname{BR}} :=m^{\sigma, \lambda}$ (respectively $m^{\operatorname{BR}}_* := m^{ \lambda, \sigma}$). These are the most interesting examples of measures invariant under the expanding (respectively contracting) horospherical groups.  We can also recover the Haar measure on $G$ as a generalized BMS measure: it is given by ${m^{\operatorname{Haar}} = m^{\lambda, \lambda}}$. 

The volume entropy of $\tilde X$ is defined in terms of the volume of large metric balls;
\[ D(\tilde X) := \lim_{R\rightarrow+\infty} \frac{\log(\mbox{vol}(B(R, o)))}{R}. \]
The limit exists and is independent of base point; it coincides with the topological entropy of geodesic flow on compact quotients of $\tilde X$. For the cases we are interested in we have explicit values of $D(\tilde X)$ as follows.

\begin{table}[h]
\centering
\begin{tabular}{l c}
\hline \hline
Group & Volume entropy\\
\hline
$\operatorname{SO}(n, 1)_+$ & $n -1$ \\
$\operatorname{SU}(n, 1)$& $2n$ \\
$\operatorname{Sp}(n, 1) $& $4n + 2$ \\
$\operatorname{F}_4^{-20} $& $22$ \\
\end{tabular}
\end{table}

\subsection{Geometrically finite groups} General discrete subgroups $\Gamma < G$ can be very difficult to understand, so it is often useful to restrict our attention to certain subclasses whose behavior is simpler. For our purposes the most useful class is of geometrically finite groups. Let $\operatorname{CH}(\Lambda(\Gamma)) \subset \tilde X$ be the convex hull of $\Lambda(\Gamma)$ (that is, the minimal convex subset containing all geodesics joining two elements of $\Lambda(\Gamma)$). Denote by $\operatorname{CC}(\Gamma)$ the convex core of $\Gamma$, i.e. the projection of $\operatorname{CH}(\Lambda(\Gamma))$ to $\Gamma \backslash \tilde X$. 

\begin{definition}
We say that the discrete group $\Gamma < G$ is geometrically finite if every $\epsilon$ neighborhood of $ \operatorname{CC}(\Gamma)$ has finite (Riemannian) volume. 
\end{definition}

The significance of this definition for our purposes is that it ensures finiteness and ergodicity of the BMS measure on $\Gamma \backslash \mathrm{T}^1\tilde X$. 

\begin{theorem}[Theorem 4.1 and Corollary 5.5 of \cite{corlette1999limit}]
If $\Gamma < G$ is geometrically finite, then $|m^{\operatorname{BMS}}| < \infty$ and the $A$ action on $(\Gamma \backslash G /M, m^{\operatorname{BMS}})$ is ergodic. 
\end{theorem}

Geometrical finiteness of $\Gamma$ also gives a decomposition of $\Lambda(\Gamma)$ into the radial limit set and the parabolic limit set. We say that a point $\xi \in \Lambda(\Gamma)$ is parabolic if the stabilizer $\Gamma _\xi :=  \mbox{stab}_\Gamma(\xi)$ is parabolic (i.e. if the fixed point set of $\Gamma_\xi$ in $\partial \tilde X$ is exactly $\lbrace \xi \rbrace$). We say that $\xi$ is bounded parabolic if it is parabolic and if the action of $\Gamma_\xi$ on $\Lambda(\Gamma) \setminus \lbrace \xi \rbrace$ is cocompact. We write $\Lambda_p(\Gamma)$ for the set of bounded parabolic limit points of $\Gamma$. We say that $\xi \in \Lambda(\Gamma)$ is a radial limit point if every geodesic ray $x(t)$ in $\tilde X$ ending at $\xi$ has some $r$ neighborhood which meets $\Gamma o$ in an infinite number of points. We write $\Lambda_r(\Gamma)$ for the radial limit set. 

\begin{theorem}[See \cite{bowditch1995geometrical}]
If $\Gamma < G$ is geometrically finite, then the limit set is the disjoint union of the radial and bounded parabolic limit sets. Moreover there are finitely many $\Gamma$ orbits of parabolic limit points. 
\end{theorem}
Note that, if $\Gamma$ is geometrically finite, then any parabolic limit point of $\Gamma$ is bounded parabolic (since it can't be radial). Assuming that $\Gamma$ is geometrically finite also allows us to choose disjoint collections of horoballs for each parabolic fixed point. By a horoball based at $\eta \in \partial \tilde X$ we mean a set of the form $\lbrace x \in \tilde X:  \beta_\eta(x, o) < T \rbrace$. 

\begin{lemma}[See, for example,  Lemma 1 of Proposition 1.10 in \cite{roblin_ergodicite_2003}]
Suppose that $\Gamma < G$ is geometrically finite and that $\xi \in \partial \tilde X$ is a parabolic fixed point. There is a collection of horoballs $\lbrace B_{\eta}: \eta \in \Gamma \xi\rbrace$ such that 
\begin{itemize}
\item{ $B_\eta$ is based at $\eta$, }
\item{ the sets $B_\eta$ are pairwise disjoint, and }
\item{ $\gamma B_\eta = B_{\gamma \eta}$ for all $\gamma \in \Gamma$ and $\eta \in \Gamma \xi$.}
\end{itemize}
\end{lemma}

We take this opportunity to fix one further piece of notation: we say that an element $g \in G$ is loxodromic if its fixed point set (in $\tilde X \cup \partial \tilde X$) consists of exactly two boundary points. One of those fixed points will be attracting, and one repelling, and $g$ acts as a translation along the geodesic joining the two. If $\Gamma$ is non elementary, then the fixed points of loxodromic elements are dense in the limit set.

\section{The transitivity group} \label{sec3} \label{sec4}

This section is the heart of the proof of Theorem \ref{target}. We begin by defining the transitivity groups and investigating their basic properties in subsection \ref{definetransgp}. The rest of the section is then concerned with showing that the transitivity groups are dense in $M \times A$ whenever $\Gamma < G$ is Zariski dense. The main result is Theorem \ref{transitivitygroupsaredense}.

\subsection{Defintion and basic properties of the transitivity group.} \label{definetransgp}

\begin{definition} Fix $g\in \mathrm{supp}(m^{\operatorname{BMS}})$. We will say that $g'  \in \mathrm{supp}(m^{\operatorname{BMS}}) $ is reachable from $g$ if there is a sequence $h_i \in N^- \cup N^+, i = 1 \ldots k$ and a $\gamma \in \Gamma$ such that
\[ gh_1 h_2, \ldots h_r \in \mathrm{supp}(m^{\operatorname{BMS}}) \mbox{ for all $0 \leq r \leq k$, and } \]
\[ \gamma gh_1 h_2\ldots h_k = g'. \]
\end{definition}

\begin{remark}
For lattices the BMS measure has full support, so the first condition can be ignored.
\end{remark}
For $g \in  \mathrm{supp}(m^{\operatorname{BMS}})$ we denote by
\[ \mathcal{H}_\Gamma(g) := \lbrace (m, a) \in M\times A: gma \mbox{ is reachable from } g \rbrace.\]
\begin{lemma} \label{l3.1} For every $g \in  \mathrm{supp}(m^{\operatorname{BMS}})$, the subset $\mathcal{H}_\Gamma(g) \subset MA$ is a subgroup; we will call it the transitivity group of $\Gamma$ at $g$. 
\end{lemma}

\begin{remark}
Brin gave a related definition \cite{brin_ergodic_1979} for compact manifolds, and Flaminio-Spatzier generalized that idea to geometrically finite hyperbolic manifolds \cite{flaminio_geometrically_1990}; their notion corresponds to our ``strong transitivity group'', which we will discuss shortly.
\end{remark}

\begin{proof}
We first check that the transitivity group contains inverses. Suppose that $(m, a) \in \mathcal{H}_\Gamma(g)$. Take a sequence $h_1 \ldots h_k \in N^- \cup N^+$ and $\gamma \in \Gamma$ for $ma$, i.e.
\[ gh_1 h_2, \ldots h_r \in \mathrm{supp}(m^{\operatorname{BMS}}) \mbox{ for all $0 \leq r \leq k$, and } \]
\[ \gamma gh_1 h_2\ldots h_k = gma. \]
The conjugate inverse sequence $(ma) h_k^{-1} (ma)^{-1} ,\ldots (ma) h_{1}^{-1} (ma)^{-1}$ satisfies
\begin{eqnarray*}  g ma h_k^{-1} \ldots h_r^{-1} (ma)^{-1} &=&\gamma g h_1 \ldots h_k  h_k^{-1} \ldots h_r^{-1} (ma)^{-1}  \\
&=&\gamma  gh_1 \ldots h_{r-1} (ma)^{-1} \end{eqnarray*}
which is in supp$(m^{\operatorname{BMS}})$ because $gh_1 \ldots h_{r-1}$ is. Similarly
\begin{eqnarray*} \gamma^{-1} g ma h_k^{-1} \ldots h_1^{-1} (ma)^{-1} &=& g h_1 \ldots h_k  h_k^{-1} \ldots h_1^{-1} (ma)^{-1}  \\
&=& g(ma)^{-1}. \end{eqnarray*}

Next we check that products are contained in the transitivity group. Suppose that $ma$ and $\tilde m\tilde a$ are both in $\mathcal{H}_\Gamma(g)$ and take admissible sequences $h_i, \tilde h_i \in N^-\cup N^+$ and  $\gamma, \tilde \gamma \in \Gamma$ such that 
\[ g m a = \gamma g h_1 \ldots h_k \mbox{ and } g  \tilde m \tilde a = \tilde \gamma g \tilde h_1 \ldots \tilde h_l.\]
We claim that $\tilde h_1 \ldots \tilde h_l, (\tilde m\tilde a)^{-1} h_1 \tilde m\tilde a, \ldots (\tilde m \tilde a)^{-1} h_k \tilde m\tilde a$ is a sequence for $ma\tilde m \tilde a$ together with $\gamma \tilde \gamma$. The argument is similar to the one above. 

\end{proof}

\begin{lemma}
If $ \gamma_0 \in \Gamma$, $u \in N^- \cup N^+ \cup A$, and if $g, \gamma_0 g u \in \mathrm{supp}(m^{\operatorname{BMS}})$, then $\mathcal{H}_\Gamma(g) = \mathcal{H}_\Gamma(\gamma_0 g u)$ as subgroups of $M\times A$. 
\end{lemma}

\begin{proof}
Suppose that $h_1 \ldots h_k$ and $\gamma$ is a sequence for $  m a\in \mathcal{H}_\Gamma(g)$. Then 

\begin{itemize}

\item{ $h_1, \ldots h_k$ and $\gamma_0 \gamma \gamma_0^{-1}$ is a sequence for $  ma$ in $\mathcal{H}_\Gamma(\gamma_0g)$,}
\item{ ${u}^{-1} , h_1, \ldots h_k,  (am)^{-1}{u} (am) $ and $ \gamma $ is a sequence for $  ma$ in $\mathcal{H}_\Gamma(g u)$ whenever $ u \in N^-\cup N^+$, and}
\item{ ${u}^{-1} h_1{u},  \ldots ,{u}^{-1} h_k {u} $ and $ \gamma $ is a sequence for $  ma$ in $\mathcal{H}_\Gamma(g u)$ whenever $u \in A$.}

\end{itemize}

\end{proof} 

\begin{lemma}
If $m \in M$ and $g \in  \mathrm{supp}(m^{\operatorname{BMS}})$, then 
\[\mathcal{H}_\Gamma(g  m) =  m^{-1} \mathcal{H}_\Gamma(g)  m.\]
\end{lemma}

\begin{proof}
Suppose that $h_1 \ldots h_k$ and $\gamma$ is a sequence for $\tilde m\tilde a \in \mathcal{H}_\Gamma(g)$.  Then $  m^{-1} h_1 m , \ldots  m^{-1} h_k m$ and $\gamma$ is a sequence for $ m^{-1} \tilde m \tilde a  m$ in $\mathcal{H}_\Gamma(g m)$.
\end{proof}

\begin{remark}
We have just seen that the transitivity group $\mathcal{H}_\Gamma(g)$ depends very little on which $g$ we choose. In particular, the question we are most interested in, whether $\mathcal{H}_\Gamma(g) < MA$ is dense, is entirely independent of the choice of $g \in \mathrm{supp}(m^{\operatorname{BMS}})$. For this reason we will often refer to the transitivity group without reference to its base point. Note that the transitivity group does depend very much on the discrete subgroup $\Gamma$.
\end{remark}

With these basics in hand we define the strong transitivity group 
\[\mathcal{H}_\Gamma^s(g)  = (  M \times \lbrace e \rbrace)  \cap \mathcal{H}_\Gamma(g)\]
and the weak transitivity group $\mathcal{H}_\Gamma^w(g)$, which is the projection $\pi_M(\mathcal{H}_\Gamma(g))$ of the transitivity group to the $M$ coordinate. Observe that 
\[\mathcal{H}_\Gamma^s(g) < \mathcal{H}_\Gamma^w(g) < M.\] 

To finish this section we will use the algebraic properties of the Bruhat decomposition to show that the transitivity group is reasonably large. Let 
\[ W = \lbrace g \in G \mbox{ such that } g^+ \neq v_0^- \rbrace\]
and let
\[ \zeta = \zeta_{N^+} \times \zeta_M \times \zeta_A \times  \zeta_{N^-} : W \rightarrow  N^+ \times M \times A \times N^- \]
be the inverse to the product map as before. 

\begin{lemma} \label{lemma3.3} Let  $g \in \mathrm{supp}(m^{\operatorname{BMS}})$. Then $\zeta_M(g^{-1}\Gamma g\cap W) \subset \mathcal{H}_\Gamma^w(g)$. 
\end{lemma}

\begin{proof}
Suppose that $g^{-1} \gamma g = h ma  n $ (with $h \in N^+, m \in M, a\in A, n \in N^-$). Then $\gamma g = g (g^{-1} \gamma g) = gh ma n $. Thus $ n^{-1},(ma)^{-1}  h^{-1}ma$, and $\gamma$ give a sequence for $ma$ in $\mathcal{H}_\Gamma(g)$.
\end{proof}

\begin{corollary} \label{wtgd}
Let  $g \in \mathrm{supp}(m^{\operatorname{BMS}})$. If $\Gamma <G$ is Zariski dense, then the weak transitivity group $\mathcal{H}_\Gamma^w(g)$ is dense in $M$.
\end{corollary}

\begin{proof}
This follows immediately from Lemma \ref{lemma3.3} and Proposition \ref{dense in M}. 
\end{proof}

\begin{remark}
In the real hyperbolic case, Corollary \ref{wtgd} follows from the arguments of Flaminio-Spatzier \cite{flaminio_geometrically_1990}. This is enough to prove ergodicity of the frame flow; mixing of frame flow, however, requires density of the transitivity group $\mathcal{H}_\Gamma(g)$ in $MA$, which does not follow immediately from Corollary \ref{wtgd} unless $M$ is semisimple.

\end{remark}

\begin{corollary} \label{transitivitygropuisnotsmall} 
Let  $g \in \mathrm{supp}(m^{\operatorname{BMS}})$. If $\Gamma< G$ is Zariski dense, then the closure of the strong transitivity group contains the commutator $[M, M]$.
\end{corollary}

\begin{proof}
For $m_i \in \mathcal{H}_\Gamma^w(g)$ and $(m_i, a_i ) \in \mathcal{H}_\Gamma(g)$ the commutator lands in $\mathcal{H}_\Gamma^s(g)$:
\[ [(m_1, a_1), ( m_2, a_2)] = ( [m_1, m_2], 0) \in \mathcal{H}_\Gamma^s(g).\]
It follows that $\overline{ \mathcal{H}_\Gamma^s(g)}  > [ \overline{ \mathcal{H}_\Gamma^w(g)}, \overline{ \mathcal{H}_\Gamma^w(g)}]   = [M, M].$
\end{proof}

\begin{remark}
If $M$ is semisimple, then this fact, together with non-arithmeticity of the length spectrum of $\Gamma$, is enough to prove density of $\mathcal{H}_\Gamma(g) <MA$. In the case that $M$ has non-trivial center, however, a more careful approach is needed.
\end{remark}

\subsection{Smooth constraints on the limit set.}  \label{useful equation} Our next target is Corollary \ref{inasmoothsubmanifold}. In order to motivate our approach, however, we will start with an example, 
\begin{example} Let $G = \operatorname{PSL}_2(\mathbb{R}) = \operatorname{Isom}_+(\mathbb{H}_\mathbb{R}^2)$. We pick the natural base point $i \in \mathbb{H}_\mathbb{R}^2$, and the natural (upward pointing) unit tangent vector. ÊThen $A < G$ is  the diagonal subgroup and $M$ is trivial. A short calculation gives
\[ N^+ = \left\lbrace \left( \begin{array}{cc} 1 & 0 \\ x& 1 \end{array} \right) : x\in \mathbb{R} \right\rbrace \mbox{ and } N^- = \left\lbrace \left( \begin{array}{cc} 1 & x\\ 0& 1 \end{array} \right): x \in \mathbb{R} \right\rbrace . \]
We also identify $\partial \mathbb{H}^2_\mathbb{R}$ with $\mathbb{R} \cup \lbrace \infty \rbrace$ in the usual way. 
Assume that $\Gamma < \operatorname{PSL}_2(\mathbb{R})$ is non-elementary and that ${\lbrace 0, 1, \infty \rbrace \subset \Lambda(\Gamma)}$. Then $e \in \mathrm{supp}(m^{\operatorname{BMS}})$, and the transitivity group $\mathcal{H}_\Gamma(e)$ is dense in $MA  = A $.
\end{example} \label{usefulequation}
\begin{proof} The first claim is clear. For the second, we define a smooth map $f : \partial( \mathbb{H}^2) \setminus \lbrace 0, \infty \rbrace  \rightarrow A$ by sending $\xi$ to $g_6$ as in the following picture.\\

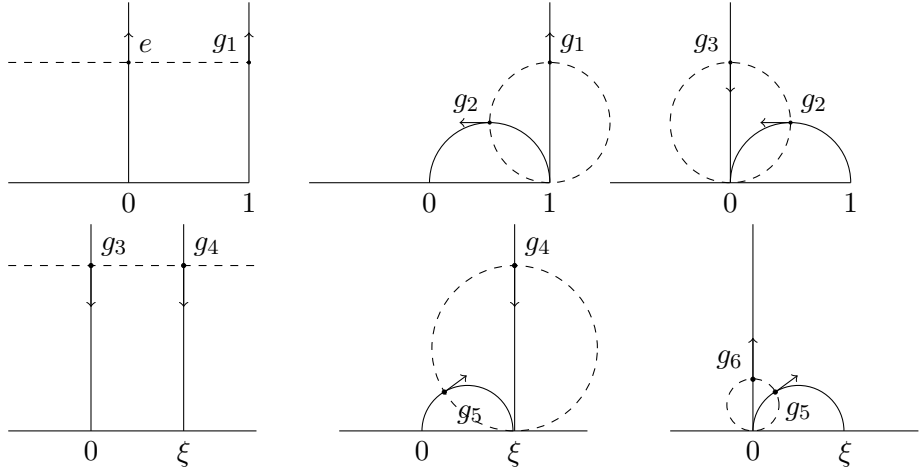
\begin{tikzpicture} [scale = 0.8]

\draw (0, 0) -- (4, 0);
\draw[dashed] (0, 2) -- (4, 2);
\draw (2, 0) -- (2, 3);
\draw (4, 0) -- (4, 3);
\draw[fill] (4,2) circle [radius=0.025];
\node [above left] at (4,2) {$g_1$};
\draw[fill] (2,2) circle [radius=0.025];
\node [above right] at (2,2) {$e$};
\node[below] at (2, 0) {$0$};
\node[below] at (4, 0) {$1$};
\draw [->] (4,2) -- (4,2.5);
\draw [->] (2,2) -- (2,2.5);

\draw (5, 0) -- (9, 0);
\draw (9, 0) -- (9, 3);
\draw (7,0) arc (180:0:1) ;
\draw [dashed] (10,1) arc (360:0:1) ;
\draw[fill] (8,1) circle [radius=0.025];
\node [above left] at (8,1) {$g_2$};
\draw[fill] (9,2) circle [radius=0.025];
\node [above right] at (9,2) {$g_1$};
\node[below] at (7, 0) {$0$};
\node[below] at (9, 0) {$1$};
\draw [->] (9,2) -- (9,2.5);
\draw [->] (8,1) -- (7.5,1);

\draw (10, 0) -- (14, 0);
\draw (12, 0) -- (12, 3);
\draw (12,0) arc (180:0:1) ;
\draw [dashed] (13,1) arc (360:0:1) ;
\draw[fill] (12,2) circle [radius=0.025];
\node [above left] at (12,2) {$g_3$};
\draw[fill] (13,1) circle [radius=0.025];
\node [above right] at (13,1) {$g_2$};
\node[below] at (12, 0) {$0$};
\node[below] at (14, 0) {$1$};
\draw [->] (12,2) -- (12,1.5);
\draw [->] (13,1) -- (12.5,1);

\end{tikzpicture}

\begin{tikzpicture} [scale = 1.1]
\draw (8, 0) -- (11, 0);
\draw (9, 0) -- (9, 2.5);
\draw (9,0) arc (180:0:0.55) ;
\draw [dashed] (9, 0.315) circle [radius=.315]; 
\draw[fill] (9,0.625) circle [radius=0.025];
\node [above left] at (9,0.625) {$g_6$};
\draw[fill] (9.27, 0.47) circle [radius=0.025];
\node [below right] at (9.27,0.47) {$g_5$};
\node[below] at (9, 0) {$0$};
\node[below] at (10.12, 0) {$\xi$};
\draw [->] (9,0.625) -- (9,1.125);
\draw [->] (9.27, 0.47) -- (9.54, 0.67);

\draw (4, 0) -- (7, 0);
\draw (6.12, 0) -- (6.12, 2.5);
\draw (5,0) arc (180:0:0.55) ;
\draw [dashed] (6.12,1) circle [radius=1] ;
\draw[fill] (6.12,2) circle [radius=0.025];
\node [below left] at (5.866,0.866/2) {$g_5$};
\draw[fill] (5.27,0.47) circle [radius=0.025];
\node [above right] at (6.12,2) {$g_4$};
\node[below] at (5, 0) {$0$};
\node[below] at (6.12, 0) {$\xi$};
\draw [->] (5.27, 0.47) -- (5.54,0.67);
\draw [->] (6.12, 2) -- (6.12, 1.5);

\draw (0, 0) -- (3, 0);
\draw (2.12, 0) -- (2.12, 2.5);
\draw (1,0) --(1, 2.5) ;
\draw [dashed] (0,2) --(3, 2) ;
\draw[fill] (2.12,2) circle [radius=0.025];
\draw[fill] (1,2) circle [radius=0.025];
\node [above right] at (2.12,2) {$g_4$};
\node [above right] at (1,2) {$g_3$};
\node[below] at (1, 0) {$0$};
\node[below] at (2.12, 0) {$\xi$};
\draw [->] (2.12,2) -- (2.12,1.5);
\draw [->] (1, 2) -- (1, 1.5);

\end{tikzpicture}

Algebraically this is just the map
\begin{eqnarray*}F:  \xi \negthickspace  &\mapsto& \left( \begin{array}{cc} 1 & 1\\0&1 \end{array}\right) \left( \begin{array}{cc} 1 & 0\\-1&1 \end{array}\right)\left( \begin{array}{cc} 1 & 1\\0&1 \end{array}\right) \left( \begin{array}{cc} 1 & 0\\-\xi &1 \end{array}\right) \left( \begin{array}{cc} 1 & \xi^{-1} \\0&1 \end{array}\right)\left( \begin{array}{cc} 1 & 0\\-\xi&1 \end{array} \right) \\
\negthickspace&=& \left( \begin{array}{cc} -\xi & 0\\0& - \xi^{-1} \end{array}\right)\\
\negthickspace&=& \left( \begin{array}{cc} \xi & 0\\0& \xi^{-1} \end{array}\right). \end{eqnarray*}
The points are that this map is a local diffeomorphism, that it carries ${1 \in \partial (\mathbb{H}^2)} $ to the identity in $A\simeq\mathbb{R}$, and that it carries the limit set into the transitivity group. Since $\Gamma$ is not elementary we know that $1$ is not isolated in the limit set, and so that $e$ is not isolated in the transitivity group. But a subgroup of $\mathbb{R}$ where the identity is not isolated is necessarily dense. 
\end{proof}

\begin{remark} For $G = \operatorname{SL}_2(\mathbb{R})$, this argument is sufficient to prove that the transitivity groups contain $M = \lbrace  \pm \operatorname{Id} \rbrace$. Together with the results of the next two sections that suffices to prove mixing of "frame flow" for $\operatorname{SL}_2(\mathbb{R})$ in the sense of  Theorem \ref{target}.
\end{remark}

We now aim to prove density of the transitivity groups for general $G$ and $\Gamma <G$ Zariski dense. The strategy is roughly the same as that in the example above; construct a map from the limit set to the transitivity group and then use smoothness properties of that map and density properties of the limit set to conclude that the transitivity group is dense. Unfortunately, we can't write our map down explicitly in the general case. Let $\tilde M = [M, M]$ be the commutator subgroup of $M$ and $M^{\mathrm{ab}} = M/\tilde M$ be the abelianization. We will write $\tilde \pi: M \times A  \rightarrow M^{\mathrm{ab}} \times A $ for the projection map. 
\begin{proposition} \label{nicefunction} Suppose that $\Gamma < G$ is Zariski dense. Fix $g \in \mathrm{supp}(m^{\operatorname{BMS}})$. There is a neighborhood $U$ of $g^-$ in $ \partial \tilde X$ and a smooth map
\[ \Phi : U \rightarrow M \times A \]
such that
\begin{itemize}
{\item $\Phi$ maps $\Lambda(\Gamma) \cap U$ into $\mathcal{H}_\Gamma(g)$, }
{\item $\Phi (g^-) = (e, e)$ is the identity,  and }
{\item the differential map $d (\tilde \pi \circ \Phi) : \mathrm{T}_{g^-}( \partial \tilde X) \rightarrow \mathrm{T}_e ( M^{\mathrm{ab}} \times A)$ is surjective.  }
\end{itemize}
\end{proposition}
There are two cases here depending on whether the abelianization of $MA$ has dimension one or two. We will consider the latter case and leave the (simpler) former as an exercise. 

\begin{lemma} \label{l4.2} Let $\hat{\mathfrak{g}} \subset \mathfrak{g}$ be the subspace spanned by the Lie algebras of $N^+, N^-,$ and $\tilde M$. There exist tangent vectors $V \in \mathrm{T}_eN^+$ and $Y_i \in \mathrm{T}_eN^-$ such that 
\[  \lbrace [V, Y_1], [V, Y_2] \rbrace  \mbox{ is a basis of } \mathfrak{g}/\hat{\mathfrak{g}}. \]
\end{lemma}

\begin{proof}
Fix a basis $Z_3, \ldots Z_r$ for $\hat{ \mathfrak{g}}$. We consider the map
\[ \Xi : (\mathrm{T}_eN^+)^2 \times (\mathrm{T}_eN^-)^2 \rightarrow \bigwedge ^r \mathfrak g = \mathbb{R} \]
defined by 
\[ (V_1, V_2, Y_1, Y_2) \rightarrow [V_1, Y_1] \wedge [V_2, Y_2] \wedge Z_3\wedge   \ldots \wedge Z_r.\]
Observe that this map is linear in all four variables, and that it has a symmetry
\[ \Xi(V_1, V_2, Y_1, Y_2) = - \Xi(V_2, V_1, Y_2, Y_1). \]
Note that $\Xi$ cannot be identically zero; semisimplicity of $\mathfrak{g}$ implies that 
\[ \mathfrak g = [\mathfrak g, \mathfrak g] \subset \hat{\mathfrak{g}}  + [\mathrm{T}_eN^-. \mathrm{T}_eN^+], \]
so vectors in $[\mathrm{T}_eN^-. \mathrm{T}_eN^+]$ must span $\mathfrak{g}/\hat{\mathfrak{g}}$.

Our task is to find some trio $V \in \mathrm{T}_eN^+$ and $Y_i \in \mathrm{T}_eN^-$ such that $\Xi(V, V, Y_1, Y_2)$ is non-zero. We will argue by contradiction. If no such trio exists, then we get a second symmetry
\[ \Xi(V_1, V_2, Y_1, Y_2) = - \Xi(V_2, V_1, Y_1, Y_2), \]
by linearity.

We want one more symmetry of $\Xi$. Let $k_0 \in K$ be a representative of the non-trivial element of the Weyl group. Then
\begin{itemize}
\item{ $\operatorname{Ad}_{k_0} \mathrm{T}_eA = \mathrm{T}_eA$,}
\item{ $\operatorname{Ad}_{k_0} \mathrm{T}_eM = \mathrm{T}_eM$, }
\item{ $\operatorname{Ad}_{k_0} \mathrm{T}_eN^\pm = \mathrm{T}_eN^\mp$, and }
\item{ The vectors $\operatorname{Ad}_{k_0}(Z_3) \ldots \operatorname{Ad}_{k_0}(Z_r)$ still give a basis for $\hat{\mathfrak{g}}$.} 
\end{itemize}
Now let $Y \in \mathrm{T}_eN^-$ and $V_i \in \mathrm{T}_eN^+$. Then our contradictory assumption implies that
\begin{eqnarray*} 0 &=& \Xi( \mathrm{Ad}_{k_0}Y, \mathrm{Ad}_{k_0}Y, \mathrm{Ad}_{k_0}V_1, \mathrm{Ad}_{k_0}V_2)\\
&=& \mathrm{Ad}_{k_0}([Y, V_1]) \wedge \mathrm{Ad}_{k_0}([Y, V_2]) \wedge Z_3\wedge \ldots  \wedge Z_r \\
&=& \mathrm{Ad}_{k_0}([Y, V_1]) \wedge \mathrm{Ad}_{k_0}([Y, V_2]) \wedge \mathrm{Ad}_{k_0} Z_3 \wedge \ldots \wedge \mathrm{Ad}_{k_0}  Z_r \\
&=& [Y, V_1] \wedge [Y, V_2] \wedge Z_3 \wedge \ldots  \wedge Z_r \\
&=& \Xi(V_1, V_2, Y, Y).\end{eqnarray*}
This gives our final symmetry $\Xi(V_1, V_2, Y_1, Y_2) = - \Xi(V_1, V_2, Y_2, Y_1)$. But any tensor satisfying these three symmetries is identically zero, which is a contradiction. 

\end{proof}
\begin{corollary} \label{c4.3}
Fix $Y_i$ as in Lemma \ref{l4.2}. There is a non-empty Zariski open set $L \subset N^+$ such that $\lbrace \mathrm{Ad}_{h} Y_1, \mathrm{Ad}_{h} Y_2 \rbrace  \mbox{ is a basis of } \mathfrak{g}/\hat{\mathfrak{g}} \mbox{ whenever } h \in L$.
\end{corollary}
\begin{proof} Fix a basis $Z_3, \ldots,  Z_r$ for $\hat{ \mathfrak{g}}$ and consider the function 
\begin{eqnarray*} \hat \Xi : N^+ \hspace{-3mm} &\rightarrow& \bigwedge^r \mathfrak{g} \\
 h &\mapsto&  \mathrm{Ad}_{h} Y_1\wedge  \mathrm{Ad}_{h} Y_2\wedge  Z_3\wedge \ldots \wedge Z_r.\end{eqnarray*}
 The first observation is that $\hat \Xi$ is non-constant; for any curve $\alpha(t)$ passing through the identity of $N^+$ we have 
\[ \frac{ d^2}{dt^2} \hat \Xi (\alpha (t))  = \Xi (\alpha'(0), \alpha'(0), Y_1, Y_2)  \]
at $t = 0$, which is non-zero for $\alpha'(0) = V$ as in Lemma \ref{l4.2}. It follows that $\hat \Xi$ is non-constant algebraic. We choose $L$ to be the complement of the pre-image of zero. 
\end{proof}

\begin{proof}[Proof of Proposition \ref{nicefunction}] Let $L \subset N^+$ be the set from Corollary \ref{c4.3}.  Choose $h_0 \in L$ such that $(gh_0)^+$ is in the (Zariski dense) limit set $\Lambda(\Gamma)$. We define the function 
\[ \tilde \Phi: N^- \times N^+ \times N^- \rightarrow G \]
by
\[ \tilde{\Phi} (n_1, h_2, n_3) = h_0 n_1 h_0^{-1} h_2 n_3. \]
By the implicit function theorem we may choose a neighborhood $\tilde U$ of $e$ in $N^-$ and smooth functions $f_2: \tilde U \rightarrow N^+$ and $f_3: \tilde U \rightarrow N^-$ such that
\[ (\tilde {\Phi} (n_1, f_2(n_1), f_3(n_1)))^-  = v_0^- \mbox{ and } (\tilde {\Phi} (n_1, f_2(n_1), f_3(n_1)))^+  = v_0^+ \]
for all $n_1 \in \tilde U$. It follows that $\hat \Phi(n_1) :=  \tilde {\Phi} (n_1, f_2(n_1), f_3(n_1))$ is a smooth map $\tilde U \rightarrow MA$. Let
\[ U = \lbrace \xi \in \partial \tilde  X \setminus  \lbrace gh_0 v_0^+ \rbrace :n((gh_0)^{-1} \xi) \in \tilde U \rbrace,  and  \]
\[ \Phi(\xi) =  \hat \Phi (n((gh_0)^{-1} \xi)). \]
If we fix $g$ and $h_0$ and write $n_\xi := n((gh_0)^{-1} \xi)$, then we can express this explicitly as 
\[ \Phi(\xi) = h_0 .n_\xi. h_0^{-1} f_2(n_\xi) . f_3(n_\xi). \]

\hspace{7mm}
\begin{tikzpicture}[scale=1]

\draw (2,0) circle [radius= 2];
\draw[fill] (2, 2) circle [radius=0.05];
\node[above] at (2, 2) {$g^+$};
\draw[fill] (2, -2 ) circle [radius=0.05];
\node[below] at (2, -2) {$g^-$};
\draw[fill] (2 + 1.414, 1.414 ) circle [radius=0.05];
\node[above right] at (2 + 1.414, 1.414 ) {$gh_0^+$};
\draw[fill] (2+ 1.414, -1.414 ) circle [radius=0.05];
\node[below right ] at (2 + 1.414, -1.414 ) {$\xi$};
\draw (2, -2) -- (2, 2);
\draw[fill] (2, 0) circle [radius=0.05];
\node[above left] at (2, 0 ) {$g$};
\draw [->] (2, 0) -- (2, 0.5);
\draw[dashed] (2, -1) circle [radius = 1];
\draw (2, -2) arc (180:135:4.85) ;
\draw[fill] (2.4, -.08) circle [radius=0.05];
\node[above right] at (2.38, -0.11 ) {$\thickspace g_1 := gh_0$};
\node[below] at (2, -2.5 ) {with $g_1 := gh_0$};
\draw [->] (2.38, -0.11) -- (2.64, 0.39);

\draw (7,0) circle [radius= 2];
\draw[fill] (7, 2) circle [radius=0.05];
\node[above] at (7, 2) {$g^+$};
\draw[fill] (7, -2 ) circle [radius=0.05];
\node[below] at (7, -2) {$g^-$};
\draw[fill] (7 + 1.414, 1.414 ) circle [radius=0.05];
\node[above right] at (7 + 1.414, 1.414 ) {$gh_0^+$};
\draw[fill] (7+ 1.414, -1.414 ) circle [radius=0.05];
\node[below right ] at (7 + 1.414, -1.414 ) {$\xi$};
\node[below left] at (7.38, 0 ) {$g_1$};
\draw[dashed] (7.75, 0.75) circle [radius = 0.92];
\draw (7, -2) arc (180:135:4.85) ;
\draw[fill] (7.39, -.105) circle [radius=0.05];
\draw[fill] (7.83,  -.168 ) circle [radius=0.05];
\draw [->] (7.38, -0.11) -- (7.64, 0.39);
\draw [->] (7.82, -0.16) -- (7.81, 0.37);
\draw (8.414, -1.414) arc (-135:-225:2) ;
\node[below right] at (7.82, -0.16 ) {$\negmedspace g_2\negmedspace:=\negmedspace g_1n_\xi$};
\node[below] at (7, -2.5 ) {with $g_2 := g_1n_\xi$};
\end{tikzpicture}

\hspace{7mm}
\begin{tikzpicture} [scale = 1]

\draw (2,0) circle [radius= 2];
\draw[fill] (2, 2) circle [radius=0.05];
\node[above] at (2, 2) {$g^+$};
\draw[fill] (2, -2 ) circle [radius=0.05];
\node[below] at (2, -2) {$g^-$};
\draw[fill] (2 + 1.414, 1.414 ) circle [radius=0.05];
\node[above right] at (2 + 1.414, 1.414 ) {$gh_0^+$};
\draw[fill] (2+ 1.414, -1.414 ) circle [radius=0.05];
\node[below right ] at (2 + 1.414, -1.414 ) {$\xi$};
\node[left] at (2.55, -0.25 ) {$g_3$};
\draw[dashed] (2.9,- 0.9) circle [radius = 0.72];
\draw (2, 2) arc (180:225:4.85) ;
\draw[fill] (2.56, -.26) circle [radius=0.05];
\draw[fill] (2.83,  -.172 ) circle [radius=0.05];
\draw [->] (2.56, -0.26) -- (2.32, 0.19);
\draw [->] (2.82,- 0.16) -- (2.81, 0.37);
\draw (3.414, 1.414) arc (135:225:2) ;
\node[above right] at (2.82, -0.16 ) {$g_2$};
\node[below] at (2, -2.5 ) {with $g_3 := g_2h_0^{-1} f_2(n_\xi)$};

\draw (7,0) circle [radius= 2];
\draw[fill] (7, 2) circle [radius=0.05];
\node[above] at (7, 2) {$g^+$};
\draw[fill] (7, -2 ) circle [radius=0.05];
\node[below] at (7, -2) {$g^-$};
\draw[fill] (7 + 1.414, 1.414 ) circle [radius=0.05];
\node[above right] at (7 + 1.414, 1.414 ) {$gh_0^+$};
\draw[fill] (7+ 1.414, -1.414 ) circle [radius=0.05];
\node[below right ] at (7 + 1.414, -1.414 ) {$\xi$};
\node[right] at (7.55, -0.25 ) {\thickspace $g_3$};
\draw[dashed] (7, 0.8) circle [radius = 1.2];
\draw (7, 2) arc (180:225:4.85) ;
\draw[fill] (7.56, -.26) circle [radius=0.05];
\draw[fill] (7,  -.4 ) circle [radius=0.05];
\draw [->] (7.56, -0.26) -- (7.32,0.19);
\draw [->] (7, -0.4) -- (7, 0.11);
\draw (7, -2) -- (7, 2);
\node[below left] at (7, -0.4 ) {$ g\Phi(\xi)$};
\node[below] at (7, -2.5 ) {with $g\Phi(\xi)= g_3f_3(n_\xi)$};
\end{tikzpicture}

\noindent We can now check the required properties directly. The subset $U$ is indeed a neighborhood of $g^-$ in $\partial \tilde X$. The function $\Phi$ is a composition of smooth maps, so is smooth. If $\xi \in \Lambda(\Gamma)\cap U$, then 
\[ h_0, n_\xi, h_0^{-1} f_2(n_\xi),   f_3(n_\xi)\]
and $e\in\Gamma$ form a sequence for $\Phi(\xi)$ in $\mathcal{H}_\Gamma(g)$ (see the picture). If $Y_i \in \mathrm{T}_eN^-$ are the tangent vectors from Lemma \ref{l4.2}, then $\lbrace  d \tilde \Phi (Y_1),  d \tilde \Phi (Y_2) \rbrace  $ spans $\mathfrak{g/\hat g}$; in fact 
\[ d \tilde \Phi(Y_i) = \operatorname{Ad}_{h_0}(Y_i) \mbox { in } \mathfrak{g/\hat g}.\]
It follows that $d(\tilde \pi \circ \Phi): \mathrm{T}_{g^-}(\partial \tilde X) \rightarrow \mathrm{T}_e (M^{\mathrm{ab}}A)$ is surjective as required. 
\end{proof}

Proposition \ref{nicefunction} is useful because of the following corollary.

\begin{corollary} \label{inasmoothsubmanifold}
Suppose that $\Gamma <G$ is Zariski dense and there exists $g \in \mathrm{supp}(m^{\operatorname{BMS}})$ such that $\tilde \pi(\mathcal{H}_\Gamma(g))<  M^{\mathrm{ab}} \times A$ is not dense. Then there is a neighborhood $U'$ of $g^-$ in $\partial \tilde X$ and a smooth embedded submanifold $S \subset U'$ of codimension one such that ${\Lambda (\Gamma) \cap U' \subset S}$. 
\end{corollary}

\begin{proof} We consider the projection $H = \tilde \pi(\mathcal{H}_\Gamma(g))$ of the transitivity group to the abelianization $M^{\mathrm{ab}}A$. If $H$ is not dense, then $\overline{H} < M^{ab}A$ is a proper Lie subgroup, so is a smooth embedded submanifold of $M^{\mathrm{ab}}A$ with positive codimension. Now apply the pre-image theorem to the composition $\tilde \pi \circ \Phi$ (see Proposition \ref{nicefunction});  we conclude that there is a neighborhood $U'$ of $g^-$ in $\partial \tilde X$ and an embedded submanifold $S =  (\tilde \pi \circ \Phi)^{-1}(\overline{H}) \cap U'$ with $\Lambda(\Gamma) \cap U' \subset S$ as required. \end{proof}

\begin{remark}
There's a reasonable question of whether we can extend this corollary to show that $\Lambda(\Gamma)$ is globally contained in a smooth submanifold of the boundary. The next argument will show that we can.
\end{remark}

\subsection{Algebraic constraints on $\Lambda(\Gamma)$: an example.}
Our next task is to show that the the projection ${\tilde \pi(\mathcal{H}_\Gamma(g)) < M^{\operatorname{ab}} \times A}$ is dense whenever $\Gamma < G$ is Zariski dense and $g \in \mathrm{supp}(m^{\operatorname{BMS}})$. We seek a contradiction with the outcome of Corollary \ref{inasmoothsubmanifold}. Begin again with an example. 
\begin{example} Consider $G = \operatorname{PSL}_2(\mathbb{C})$ acting as the isometry group of real hyperbolic three space $\mathbb{H}^3_\mathbb{R}$. Identify the boundary $\partial (\mathbb{H}^3_\mathbb{R})$ with $\mathbb{R}^2 \cup \lbrace \infty \rbrace$ in the usual way. Suppose that $\Gamma < G$ is discrete. If 
\[ a_{\log 2} =   \left(\begin{array}{cc} 2 & 0 \\ 0 & \frac{1}{2} \end{array} \right)  \in \Gamma, \]
then $(0, 0) \in \Lambda(\Gamma)$. If in addition there is a smooth curve $S \subset \mathbb{R}^2$ passing through $(0, 0)$, and a positive $\epsilon$ such that
\[ \Lambda(\Gamma) \cap \lbrace (x, y)  \in \mathbb{R}^2: |(x, y)| < \epsilon \rbrace \subset S, \]
then $\Gamma$ is not Zariski dense in $\operatorname{PSL}_2(\mathbb{C})$, considered as a real algebraic group. 
\end{example}
It is clear that $(0, 0) \in \Lambda (\Gamma)$. Shrinking $\epsilon$ if necessary, we may assume that $S$ is given either as the graph of a smooth real function $y = f(x)$ or as the graph of a smooth real function $x = g(y)$. For simplicity assume the former. There is, therefore, some constant $C > 0$ with  
\[ |y - f'(0) x| \leq   Cx^2 \]
for all pairs $(x, y) \in S$ with $|(x, y)| < \epsilon$, and so for all sufficiently small pairs $(x, y) \in \Lambda(\Gamma)$. 

\begin{lemma} If $(x_0, y_0) \in \Lambda(\Gamma) \setminus \lbrace \infty \rbrace $, then $y_0 - f'(0)x_0 = 0$. \end{lemma} 
\begin{proof}  Since $ a_{\log 2} \in \Gamma$, and since the limit set is $\Gamma$ invariant, we have 
\[ a_{\log 2}^{-k}  (x_0, y_0)  =(x_0/2^k, y_0/2^k)  \in \Lambda(\Gamma)\]
for all $k$. For $k$ large enough we have that $|(x_0 / 2^k, y_0/2^k)| < \epsilon$, and so that 
\[  \frac{ |y_0 - f'(0) x_0|}{2^k} \leq C\left( \frac{x_0}{2^k} \right) ^2.  \]
But this implies that $|y_0 - f'(0) x_0| < C x_0^2/2^k$ for all $k$, which in turn implies that $y_0 - f'(0)x_0 = 0$ as expected. 
\end{proof}

From the lemma it follows that $\Lambda (\Gamma) \cap \mathbb{R}^2$ is contained in some straight line, so is not Zariski dense. It also follows that the orbit $\Gamma (0, 0)  \subset \partial (\mathbb{H}^3)$ is not Zariski dense, and so that $\Gamma < \mathrm{PSL}_2(\mathbb{C})$ is not Zariski dense. This completes the example. 

\subsection{Algebraic constraints on $\Lambda(\Gamma)$: the general case.}
The example above actually contains all the necessary ideas to prove our next proposition. We now repeat the argument in a more general setting and with fewer simplifying assumptions. 
\begin{proposition} \label{algcon}
Suppose that there is a loxodromic element $\gamma$ of $\Gamma$ whose expanding fixed point is $v_0^-$ and whose contracting fixed point is $v_0^+$ (so \\ ${\gamma  = ma \in MA}$). Suppose further that there exists a neighborhood $U$ of $v_0^-$ in $\partial \tilde X$, and a smooth submanifold $S \subset U$ of codimension one, such that $\Lambda(\Gamma) \cap U \subset S$. Then $\Gamma < G$ is not Zariski dense. 
\end{proposition}

\begin{proof} 
We identify $N^-$ with $ \partial \tilde X \setminus \lbrace v_0^+ \rbrace$ as usual by taking $n$ to its negative end point  $n^-$. We further identify $\mathrm{T}_eN^-$ and $N^-$ via the exponential map. The action of $\gamma$ on $\partial \tilde  X$ corresponds to the action of $\gamma$ on $N^-$ by conjugation, and so to the adjoint action of $\gamma$ on $\mathrm{T}_eN^-$. The action of $a$ on $\mathrm{T}_eN^-$ is diagonalizable over $\mathbb{R}$, so take an eigenbasis $v_1, \ldots v_k$ with eigenvalues $\alpha_i$. We choose the corresponding coordinates $x_i$ on $\mathrm{T}_eN^-$. Since $v_0^-$ is an expanding fixed point we know that $\alpha := \min \lbrace \alpha_i \rbrace  > 1$. Since $m$ commutes with $a$, the $m$ action on $\mathrm{T}_eN^- = \mathbb{R}^k$ respects the eigenspaces of $a$. 

Now consider the hypersurface $S\subset \mathrm{T}_eN^-$. We know that $0 \in S$. Shrinking $U\subset \mathrm{T}_eN^-$ if necessary, we may assume that there is a coordinate $x_j$ such that $S$ is the graph
\[ S = \lbrace (x_1, \ldots x_k) \in\mathbb{R}^k \mbox{ such that } x_j = f(x_1, \ldots \hat x_j \ldots x_k ) \rbrace  \]
of a smooth function $f$. Shrinking $U$ again we may assume that $U$ is pre-compact, and that $f$ extends smoothly to the closure $\overline U$. Choose $l$ such that $\alpha^{l+1} > \alpha_j$ and let $ p(x_1, x_2, \ldots \hat x_j, \ldots x_k)$ be the degree $l$ Taylor polynomial for $f$. There is a constant $C > 0$ such that
\[  |f(x_1, \ldots  \hat  x_j \ldots x_k )   - p(x_1, \ldots \hat   x_j \ldots x_k)| < C | (x_1, \ldots \hat  x_j \ldots x_k)|^{(l+1)} \]
for all $x \in U$. Now $a$ acts on the set of polynomials in the $x_i$. The action is diagonalizable, so let $p_0, p_1, \ldots p_w$ be the eigenfunction decomposition of $ x_j - p$ with eigenvalues $0 < \beta_0 < \ldots < \beta_w$. Note that $\beta_0 \leq \alpha_j$. As in the example we will show that $p_0$ vanishes on $\Lambda(\Gamma)$. To this end fix $y = (y_1 \ldots y_k) \in\mathrm{T}_eN^- \cap \Lambda(\Gamma)$. Then $\gamma^{-r}y \in \Lambda(\Gamma)$ for all $r$, and $ \gamma ^{-r} y \in \Lambda(\Gamma) \cap U \subset S$ once $r$ is sufficiently large. It follows that 
\[ |\sum_0^w p_i(\gamma^{-r} y)|\leq C|\gamma^{-r}y|^{l + 1} \]
for all large $r$. Thus
\[ |\beta_0^{-r}p_0(m^{-r} y)| \leq | \sum_1^w \beta_i^{-r} p_i(m^{-r} y)|+  C\alpha^{l+1}|m^{-r}y|^{l + 1} \]
for all large $r$. Rearranging we obtain
\[ |p_0(m^{-r} y)|  \leq  |\sum_1^w \left( \frac{\beta_i}{\beta_0}\right)^{-r} p_i(m^{-r} y)|+ C\left( \frac{ \alpha^{(l+1)} }{\beta_0} \right)^{-r} |m^{-r}y|^{l + 1}. \]
Now choose a subsequence $r_j \rightarrow \infty$ such that $m^{-r_j} \rightarrow e$. The left hand side converges to $|p_0(y)|$, while the right is bounded by
\[ |p_0(m^{-r_j} y)| \leq \left( \frac{\beta_1}{\beta_0}\right)^{-r_j} \sum_1^w \max | p_i(M y)| +C \left( \frac{ \alpha^{(l+1)} }{\beta_0} \right)^{-r_j} \max |My|^{l + 1}\]
which converges to zero.
\end{proof}

\begin{corollary} \label{notzdense}
If there exist $\xi \in \Lambda(\Gamma)$, a neighborhood $U$ of $\xi$ in the boundary $\partial \tilde X$, and a smooth codimension one submanifold $S$ of $U$ such that ${\Lambda (\Gamma) \cap U \subset S}$, then $\Gamma < G$ is not Zariski dense. In particular, if $\Gamma < G$ is Zariski dense, then $\tilde \pi(\mathcal{H}_\Gamma(g))$ is dense in $M^{\operatorname{ab}}\times A$. 
\end{corollary}

\begin{proof} Let $\xi, U, S$ be as in the statement. Choose a loxodromic element $ \gamma_0 \in \Gamma$ whose expanding end point falls in $U$, therefore in $S$. Let $\xi_0$ and $\xi_\infty$ be the expanding and contracting end points of $ \gamma_0$. Choose an element $g \in G$ such that $g(0) = \xi_0$ and $g(\infty) = \xi_\infty$. Then $g^{-1} \Gamma g$ contains the element $\tilde \gamma_0 = g^{-1} \gamma g \in MA$ and has limit set $\Lambda(g^{-1}\Gamma g) = g^{-1} \Lambda(\Gamma)$. The sets $g^{-1}U$ and $g^{-1}S$ satisfy the conditions of the Proposition \ref{algcon}, so we see that $g^{-1}  \Gamma g$ is not Zariski dense. But then neither is $\Gamma$. The second statement follows from the first statement via Corollary \ref{inasmoothsubmanifold}. \end{proof}

\begin{theorem} \label{transitivitygroupsaredense}
If $\Gamma < G$ is Zariski dense and $g \in \mathrm{supp}(m^{\operatorname{BMS}})$, then the transitivity group $\mathcal{H}_\Gamma(g)$ is dense in $M\times A$.
\end{theorem}

\begin{proof}
Consider the principal bundle
\[
\begin{CD}
0 @>>> [M, M]    @>>> M\times A @>\tilde \pi >> M^{\mathrm{ab}} \times A @>>>0.
\end{CD}
\]
We already know that $ \overline {\mathcal{H}_\Gamma^s(g)}  >  [M, M]$. It follows that $\overline{ \mathcal{H}_\Gamma(g)}$ is a union of fibers of $\tilde \pi$. But $\tilde \pi(\mathcal{H}_\Gamma(g)) \subset M^{\mathrm{ab}} \times A$ is dense, so (using the compactness of the fibers)  $\overline{ \mathcal{H}_\Gamma(g)}$ meets every fiber of $\tilde \pi$. Thus $ \mathcal{H}_\Gamma(g)   \subset M\times A$ is dense as claimed. 
\end{proof}
We finish the section by studying the case when $\tilde X=\mathbb{H}^3_\mathbb{R}$ more carefully.

\begin{lemma} \label{H^3 case}
Suppose that $G = \operatorname{PSL}_2(\mathbb{C})$ and that $\Gamma <G$ is discrete and non-elementary. Then $\overline {\mathcal{H}_\Gamma(g)} =\overline{ \mathcal{H}^w_\Gamma(g)} \times A$.
\end{lemma}

\begin{proof} Fix $g\in \mathrm{supp}(m^{\operatorname{BMS}})$, and suppose that ${\mathcal{H}_\Gamma(g)} < M \times A$ is not dense. We identify the boundary $\partial( \mathbb{H}^3)$ with $\mathbb{C}\cup \lbrace \infty \rbrace $ as usual. It is not difficult to see that if $h \in G$ then $(hg)^\pm \in \Lambda(h\Gamma h^{-1})$, so that $hg$ is in the support of the BMS measure associated to $h\Gamma h^{-1}$. It is also straightforward to check that $\mathcal{H}_\Gamma (g) = \mathcal{H}_{h\Gamma h^{-1}} (hg)$. We may therefore assume, without loss of generality, that $\lbrace 0, 1, \infty\rbrace \subset \Lambda(\Gamma)$. The mapping $\xi \mapsto F(\xi)$ (defined near the beginning of subsection \ref{usefulequation}) extends to  local diffeomorphism carrying a neighborhood of $1\in \mathbb{C}$, to a neighborhood of the identity in $M\times A$ and carrying the limit set into the transitivity group. Non-density of the transitivity group $\mathcal{H}_\Gamma(e)$ then implies that $\Lambda(\Gamma)$ is locally contained in some smooth curve through $1 \in \mathbb{C}$. 

Conjugating $\Gamma$ again, we may now assume that there is a loxodromic element $\gamma$ of $\Gamma$ whose contracting fixed point is $0$ and whose expanding fixed point is $\infty$. We may also assume that there exist $\epsilon > 0$ and a smooth curve $\alpha(t)$ with $\alpha(0) = 0$,  $\alpha'(0) = 1$ and
\[ \lbrace \xi \in \Lambda(\Gamma) \cap \mathbb{C} : |\xi| <  \epsilon \rbrace  \subset \alpha((-1, 1)). \]
The expansion argument of Proposition \ref{algcon} now implies that $\Lambda(\Gamma) \subset \mathbb{R} \cup \lbrace \infty \rbrace  \subset \mathbb{C} \cup \lbrace \infty \rbrace$. At this stage we may apply yet another conjugation to $\Gamma$ to assume that $\lbrace 0, 1, \infty \rbrace \subset \Lambda( \Gamma) \subset \mathbb{R}\cup\lbrace \infty \rbrace $. A final application of the map $F$ together with non-isolation of points in the limit set, now implies that $A \subset \overline{ \mathcal{H}_\Gamma(e)}$.  The result follows.
\end{proof}

\section{Invariant functions for the horospherical groups}  \label{sec5}
This section contains the measure theory we need to prove mixing of frame flow. The aim is to prove Corollary \ref{C5.3}. We assume finiteness of the BMS measure throughout this section. 
\subsection{Periods of $\tilde{\Sigma}$-measurable functions.}
\begin{definition}
We write $\mathcal{B}(G)$ for the Borel sigma algebra on $G$ and define subalgebras
\[   \Sigma_\pm = \lbrace B \in \mathcal{B}(G) : B = \Gamma BN^\pm \rbrace \] 
consisting of all Borel measurable sets that are invariant under $\Gamma$ and the stable (unstable) horospherical group. We further define
\[ \tilde \Sigma = \Sigma_- \wedge \Sigma_+; \]
in other words, a Borel measurable set $B \subset G$ belongs to $ \tilde \Sigma$ if there exist $B_\pm  \in \Sigma_\pm$ such that $\tilde m^{\operatorname{BMS}}(B \triangle B_\pm ) = 0$. Note that the definition of $\tilde \Sigma$ depends on the group $\Gamma$. 
\end{definition}
Our task now is to prove that $\tilde \Sigma$ is trivial whenever the transitivity groups are dense in $M \times A$. Let $\psi$ be a bounded $\tilde \Sigma$ measurable function on $G$. We assume without loss of generality that $\psi$ is strictly invariant for the left $\Gamma$ action. Let $\psi_\pm$ be bounded $\Sigma_\pm$ measurable functions satisfying $\psi_\pm = \psi$ almost everywhere (with respect to $\tilde m^{\operatorname{BMS}}$). Let
\begin{eqnarray*} E &=& \left\lbrace gMA: \begin{array}{l}  \psi|_{gMA} \mbox{ is measurable, and }\\ \psi(gma) = \psi_+(gma) = \psi_-(gma) \\ \mbox{ for Haar almost every } ma \in MA \end{array}  \right\rbrace\\ 
&\subset& G/MA. \end{eqnarray*}
It will also be useful to identify $G/MA$ with $\partial ^2 \tilde X $ via the map ${gMA \mapsto (g^-, g^+)}$. By Fubini, $E$ has full measure in $(\partial^2 \tilde X, \sigma_o \times \sigma_o)$. For $\epsilon > 0$ let $M_\epsilon$ be the $\epsilon$ ball around the identity in $M$, and define
\[ \psi^\epsilon(g) = \int_{M_\epsilon} \int_0^\epsilon \psi(gma_t) dt dm , \mbox{ and}\]
\[ \psi_\pm ^\epsilon(g) = \int_{M_\epsilon} \int_0^\epsilon \psi_\pm(gma_t) dt dm. \]
Note that if $gMA \in E$, then the functions $\psi^\epsilon, \psi^\epsilon_\pm$ are continuous, well defined, and identical on $gMA$. Now let
\[ E^- = \lbrace \xi \in \Lambda(\Gamma) : (\xi, \eta') \in E \mbox{ for $\sigma_o$ almost every $\eta' \in \partial \tilde X \setminus \lbrace \xi \rbrace $}  \rbrace  \]
and 
\[ E^+ = \lbrace \eta \in \Lambda(\Gamma) : (\xi', \eta) \in E \mbox{ for $\sigma_o$ almost every $\xi' \in \partial \tilde X \setminus  \lbrace \eta \rbrace $}  \rbrace. \]
Another application of Fubini tells us that $E^+$ and $E^-$ both have full measure in $\partial \tilde X$. To complete our setup we define
\[ E^+_\xi  = \lbrace \eta \in \Lambda(\Gamma)\setminus \lbrace \xi \rbrace  : (\xi, \eta) \in E \rbrace , \mbox{ and}\]
\[ E^-_\eta  = \lbrace \xi \in\Lambda(\Gamma)\setminus \lbrace \eta \rbrace  : (\xi, \eta) \in E \rbrace. \]
\begin{lemma}
Let $\epsilon> 0$.  Suppose that $gMA \in E$, that $g^+ \in E^+$, that $g^- \in E^-$, and that $m_0 a_0 \in \mathcal{H}_\Gamma(g)$. Then ${\psi^\epsilon(gm_0a_0) = \psi^\epsilon(g)}$. 
\end{lemma}
\begin{proof}
Consider a sequence $gm_0a_0 = \gamma g h_1 h_2 \ldots h_k$ for $m_0a_0$ in the transitivity group. We will say that this sequence is permissible if $gh_1 \ldots h_r \in E$ for all $1 \leq r \leq k$.  In this case 
\[ \psi^\epsilon(g) = \psi^\epsilon(gh_1) = \ldots = \psi^\epsilon(gh_1 \ldots h_k) = \psi^\epsilon(\gamma g h_1 \ldots h_k) = \psi^\epsilon(gm_0a_0), \]
using invariance of $\psi_+^\epsilon$ under $N^+$ , invariance of $\psi_-^\epsilon$ under $N^-$, invariance of $\psi$ under  $\Gamma$, and the fact that all three agree on $E$.

More generally, we need a limiting argument allowing us to approximate any element of the transitivity group by a permissible sequence. Suppose that $m_0a_0 \in \mathcal{H}_\Gamma(g)$. We can choose a sequence $\gamma \in \Gamma$ and $n_1, h_1, n_2 \ldots n_k, h_k$ for $m_0a_0$ with $n_i \in N^-$ and $h_i \in N^+$. For $0 \leq i \leq k$ we define 
\[ \xi_i = (gn_1h_1 \ldots n_i h_i)^- \in \Lambda(\Gamma) \mbox{ and } \eta_i = (gn_1h_1 \ldots n_i h_i)^+ \in \Lambda(\Gamma). \]
Note that $\xi_k  = (\gamma^{-1} g)^-$ and that $\eta_k = (\gamma^{-1} g)^+$. We also observe that 
\[ n_i = n( (gn_1 h_1 \ldots n_{i-1} h_{i-1})^{-1} \xi_i) \mbox{ and } h_i = h( (gn_1 h_1 \ldots n_{i-1} h_{i-1}n_i)^{-1} \eta_i), \]
so we can reconstruct the horospherical group elements $n_i, h_i$ from the boundary points $\xi_i, \eta_i$. Now we need to do our approximation. Fix $\xi_0^{l} = \xi_0$, and $\eta_0^{l} = \eta_0$ for all $l$. We can then define sequences $\xi_1^l \in \Lambda(\Gamma)$ and $n_1^{(l)} \in N^-$ as follows:
\[ \xi_{1}^l \in E_{\eta_0^l}^- \cap E^- \mbox{ with } \xi_{1}^l \rightarrow \xi_{1} , \mbox{ and}\]
\[ n_{1}^{(l)} = n(g^{-1}\xi_{1} ^l). \]
This is possible because $\xi_1$ is in the limit set, and because full measure subsets of $\Lambda(\Gamma)$ are necessarily dense. Note that 
\begin{itemize}
\item{  $(gn_1^{(l)})^- = \xi_1^l \in E^- \subset \Lambda(\Gamma),$}
\item{  $(gn_1^{(l)})^+ = \eta_0^l  = \eta_0 \in E^+ \subset \Lambda(\Gamma),$}
\item{  $n_1^{(l)} \rightarrow n_1$ as $l \rightarrow \infty$, and }
\item{ $gn_1^{(l)} \in E$ for all $l$.}
\end{itemize}
With $\xi_1^l$ and $n_1^{(l)}$ in hand we then choose a sequence $\eta_1^l$ and $h_1^{(l)}$ such that 
\[ \eta_{1}^l \in E_{\xi_1^l}^- \cap E^- \mbox{ with } \eta_{1}^l \rightarrow \eta_{1} , \mbox{ and}\]
\[ h_1^{(l)} = h((gn_1^{(l)}) ^{-1}\eta_{1} ^l).\]
As before, we note that the forward and backward end points of $gn_1^{(l)} h_1^{(l)}$ are in the limit set, that $h_1^{(l)} \rightarrow h_1$, and that $gn_1^{(l)} h_1^{(l)} \in E$ for all $l$.

Iterating these arguments we obtain sequences $\xi_i^l, n_i^{(l)}, \eta_i^l, h_i^{(l)}$ for $1 \leq i \leq k$ such that:
\begin{itemize}
\item{ $(gn_1^{(l)} h_1^{(l)} \ldots n_i^{(l)} h_i^{(l)})^+ = \eta^l_i  \in E^+ \subset \Lambda(\Gamma) $ and 

$(gn_1^{(l)} h_1^{(l)} \ldots n_i^{(l)} h_i^{(l)})^- = \xi^l_i \in E^- \subset \Lambda(\Gamma)$ for all $i, l$;}
\item{ $(gn_1^{(l)} h_1^{(l)} \ldots  h_i^{(l)} n_{i+1}^{(l)})^+ = \eta^l_i  \in E^+ \subset \Lambda(\Gamma) $ and 

$(gn_1^{(l)} h_1^{(l)} \ldots  h_i^{(l)} n_{i+1}^{(l)})^- = \xi^l_{i+1} \in E^- \subset \Lambda(\Gamma)$ for all $i, l$;}
\item{ $\xi_k^l = (\gamma^{-1} g)^-$ and $\eta_k^l = (\gamma^{-1}g)^{+}$ for all $l$;}
\item{ $\gamma gn_1^{(l)} h_1^{(l)} \ldots  n_k^{(l)} h_k^{(l)} \in gMA$ for all $l$, so the sequence $\gamma, n_1^{(l)} \ldots h_k^{(l)}$ defines an element $m_la_l$ of the transitivity group $ \mathcal{H}_\Gamma(g)$; }
\item{ for each $l$ and $i$ we have $gn_1^{(l)} h_1^{(l)} \ldots  h_{i-1}^{(l)} n_{i}^{(l)} \in E$ and 

$gn_1^{(l)} h_1^{(l)} \ldots  n_i^{(l)} h_i^{(l)} \in E$ so the sequence above is permissible;}
\item{ $n_i^{(l)} \rightarrow n_i$ and $h_i^{(l)} \rightarrow h_i$ for all $i$, so}
\item{ $m_la_l \rightarrow m_0a_0$.} 
\end{itemize}
The lemma is now immediate. We have
\[ \psi^\epsilon(gm_0a_0) = \lim_{l\rightarrow\infty}  \psi^\epsilon(gm_la_l) = \psi^\epsilon(g)  \]
as required. 
\end{proof}

\begin{corollary}
If $(g^-, g^+) \in E \cap (E^- \times E^+)$, and if $m_0 a_0 \in \mathcal{H}_\Gamma(g)$, then $m_0a_0$ is a period of $\psi^{\epsilon}|_{gMA}$: that is,
\[ \psi^{\epsilon} (gm_0a_0 ma) = \psi^{\epsilon}(gma)\]
for all $ma \in MA$.  
\end{corollary}
\begin{proof}
Since $m_0a_0 \in \mathcal{H}_\Gamma(g)$ we have $m^{-1} m_0a_0m \in \mathcal{H}_\Gamma(gma)$. Thus 
\[ \psi^\epsilon(gma) = \psi^{\epsilon}(gmam^{-1}m_0a_0m) = \psi^{\epsilon} (gm_0a_0 ma) \]
as required. 
\end{proof}

\begin{theorem} \label{C5.3}
If $\mathcal{H}_\Gamma(g) < MA$ is dense for any $g \in \mathrm{supp}(m^{\operatorname{BMS}})$, then $\tilde \Sigma$ is trivial. 
\end{theorem}
\begin{proof}
Let $\psi$ be some bounded $\tilde \Sigma$ measurable function. Then $\psi^{\epsilon}$ is right $MA$ invariant for any $\epsilon > 0$. Thus $\psi$ is right $MA$ invariant. But the right $MA$ action on $(\Gamma  \backslash G, m^{\operatorname{BMS}})$ is ergodic \cite{corlette1999limit}, so $\psi$ is constant almost everywhere. 
\end{proof}

\section{Triviality of $\tilde \Sigma$ and mixing of the frame flow}  \label{sec6}
We now have almost all the technical pieces in place to prove Theorem \ref{target}. The last piece of the argument is to show that mixing follows from triviality of $\tilde \Sigma$. We assume finiteness of the BMS measure throughout this section.

\begin{proposition}\label{trivialisenough} If $\tilde\Sigma$ is trivial (i.e.\ if every $\tilde \Sigma$-measurable set is either null or co-null for the \textup{BMS} measure), then the $A$ action on $(\Gamma \backslash G, m^{\operatorname{BMS}})$ is mixing. 
\end{proposition}
\noindent The key point in the proof is the following lemma (Lemma  1 from \cite{babillot_mixing_2002}). 

\begin{lemma} \label{Babillotslemma}
Let $(X, \mathcal{B}, m, (a_t)_{t \in \mathbb{R}})$ be a measure preserving dynamical system where $(X, \mathcal{B})$ is a standard Borel space, $m$ a finite Borel measure and $a_t$ an action of $\mathbb{R}$ on $X$ by measure preserving transformations. Let \\${\psi \in L^2(X, m)}$ be a real valued function on $X$ such that $\int_X \psi dm = 0$.

If there exist $t_n \rightarrow \infty$ in $\mathbb{R}$ such that $\psi \circ a_{t_n}$ does not converge weakly to zero in $L^2$, then there is a sequence $s_m \in\mathbb{R}$ tending to $+\infty$ and a non-constant function $\phi \in L^2(X, m)$ such that 
\[ \psi \circ a_{s_n} \rightarrow \phi \mbox{ and }  \psi \circ a_{-s_n} \rightarrow \phi \mbox{ in the weak $L^2$ topology}. \]
\end{lemma}

\begin{proof}[Proof of Proposition \ref{trivialisenough}] Suppose that the right $A$ action is not mixing. Then we may choose a uniformly continuous function
\[ \psi : \Gamma \backslash G \rightarrow \mathbb{R} \]
such that $\int_{\Gamma \backslash G}\phi dm^{\operatorname{BMS}} = 0$ and $\psi\circ a_t$ does not converge weakly to zero in $L^2(\Gamma \backslash G, m^{\operatorname{BMS}})$. By Lemma \ref{Babillotslemma} we get a non-constant function $\phi$ and $s_n \rightarrow +\infty$ such that
\[ \psi \circ a_{s_n} \rightarrow \phi \mbox{ and }  \psi \circ a_{-s_n} \rightarrow \phi \mbox{ in the weak $L^2$ topology}. \]
Passing to a subsequence $s_{n_k}$ we may assume (by Banach-Saks) that the Cesaro averages
\[ S_R(\psi) := \frac{1}{R}  \sum_{r = 1}^{R} \psi \circ a_{s_{n_r}} \mbox{ and }  \tilde S_R(\psi) := \frac{1}{R}  \sum_{r = 1}^{R} \psi\circ {a_{-s_{n_r}}} \]
converge strongly to $\phi$ in $L^2$. Passing to a further subsequence we may assume that $S_{R_l}(\psi)$ and $\tilde S_{R_l}(\psi)$ converge almost everywhere to $\phi$. 

It is standard to show that any almost everywhere limit of $S_{R_l}(\psi)$ (respectively $\tilde S_{R_l}(\psi)$) is $N^-$ invariant  (respectively $N^+$ invariant). So on the one hand $\phi = \lim_l S_{R_l}(\psi)$ is $\Sigma_-$ measurable, while on the other $\phi = \lim_l \tilde S_{R_l}(\psi)$ is $\Sigma_+$ measurable. Thus $\phi$ is a non-constant $\tilde \Sigma$ measurable function, and  $\tilde \Sigma$ cannot be trivial. 
\end{proof}

\begin{proof}[Proof of Theorem \ref{target}]  Suppose that $\Gamma < G$ is Zariski dense and that the associated BMS measure is finite. By Theorem \ref{transitivitygroupsaredense} the transitivity group $\mathcal{H}_\Gamma(g)$ is dense for any $g \in \mathrm{supp}(m^{\operatorname{BMS}})$, so $\tilde \Sigma$ is trivial by Theorem \ref{C5.3}, and the frame flow is mixing by Proposition \ref{trivialisenough}. \end{proof}

\section{Consequences of mixing of the frame flow}  \label{sec7}
This section contains the principal applications of Theorem \ref{target}. We begin in subsection \ref{ss7.1} by recalling some measures and metrics on the horospheres. These allow us to state a number of results on equidistribution and the decay of matrix coefficients in subsection \ref{ss7.2}. Those equidistributional results are essential steps in the classification theorem for Radon measures invariant under the horospherical group, which we address in subsection \ref{ss7.3}.  In subsection \ref{ss7.4} we return to consider the case when the transitivity group is not dense in $MA$. 

Throughout this Section we assume that the BMS measure is finite on $\Gamma \backslash G$. In subsections \ref{ss7.2} and \ref{ss7.3} we make the additional assumption that the frame flow is mixing in the sense of Theorem \ref{target}. Subsection \ref{ss7.4} assumes only the finiteness of the BMS measure.

\subsection{Metrics and measures on horospheres.} \label{ss7.1} 
For each $g \in G$ we have the strong stable (unstable) horosphere ${H^-(g) = gN^-}$ (respectively \\${H^+(g) = gN^+}$). We identify $H^{\pm}(g)$ with the boundary ${\partial \tilde X \setminus \lbrace g^{\mp}\rbrace} $ as usual by sending $h \rightarrow h^{\pm}$. We denote the inverse map by
\[ P_{H^\pm(g)} : \partial \tilde X \setminus \lbrace g^{\mp}\rbrace \rightarrow H^{\pm}(g).  \]
In fact this is nothing very new. We have
\[ P_{H^+(g)}(\xi) = gh(g^{-1}\xi) \mbox{ and }  P_{H^-(g)}(\xi) = gn(g^{-1}\xi) .\]
It will also be useful to notice that, for all $m \in M$ and $a\in A$,
\[ P_{H^{\pm}(gma)}(\xi) = P_{H^{\pm}(g)}(\xi) ma.\]
Using these functions, we can convert the visual distance on the boundary to give a metric on each horosphere (see \cite[Section 1G]{roblin_ergodicite_2003}):

\[ d_H(P_H(\xi), P_H(\eta)) := e^{\frac{1}{2} [\beta_\xi(o, P_H(\xi)) + \beta_{\eta}(o, P_H(\eta))]}d_{\partial \tilde X}(\xi, \eta) \]
where $d_{\partial \tilde X}$ is the Gromov visual distance on the boundary associated to our base point $o \in \tilde X$. It is worth noting that this metric is left invariant by $G$ in the sense that 
  \[ d_{gH}(P_{gH}(g\xi), P_{gH}(g\eta))  =  d_H(P_H(\xi), P_H(\eta)) \]
  for all $g \in G$. The right $MA$ action is also straightforward:
  \[ d_{H^\pm ma_t} (g_1ma_t, g_2ma_t) =e^{\pm t} d_{H^\pm} (g_1, g_2)\]
  for all  $g_i \in H^\pm, t\in \mathbb{R}$ and $m \in M$. Our last observation for now is that stable and unstable leaves have infinite diameter for these metrics. We will write $B^{\pm}(g, r)$ for the open metric balls around $g$ in $H^{\pm}(g)$ respectively.

For $\mu$ a conformal density of dimension $\delta_\mu$ on $\partial \tilde X$, and $H^{\pm}$ a horosphere, we define the measure
\[ \int_{H^\pm} \phi d\mu_{H^{\pm}} : = \int_{\partial \tilde X}  e^{\delta_\mu\beta_{\xi}(o,  P_{H^\pm} (\xi))} \phi(P_{H^{\pm}}(\xi)) d\mu_o(\xi) .\]
A special case of this is the Patterson-Sullivan measure $\sigma_o$ which gives rise to measures $\sigma_{H^\pm}$ on the stable and unstable leaves respectively. These measures have good invariance properties for the right action of $M\times A$: for all $g_0 \in G$, all $\mu$ of dimension $\delta_\mu$, and all $ m \in M, a_t \in A$, one has that
\[ e^{\pm \delta_\mu t} \int_{H^{\pm}(g_0)}\Phi(gma_t) d\mu_{H^{\pm}(g_0)}(g) =   \int_{H^{\pm}(g_0ma_t )}\Phi(g) d\mu_{H^{\pm}(g_0ma_t)}(g). \]
Also note that
\[ \gamma_* \mu_{H^{\pm}} = \mu_{\gamma H^{\pm}}. \]

\subsection{Mixing, equidistribution, and the  decay of matrix coefficients.} \label{ss7.2}
Roblin's thesis demonstrates how to use mixing of the $A$ action to prove a variety of other results on equidistribution and decay rates of matrix coefficients. We will be particularly interested in three examples, which we now state. Each of these results follows from mixing of frame flow using the arguments of \cite{roblin_ergodicite_2003}.
 \begin{theorem}  \label{R3.2} Let $ \phi \in C_c(\Gamma \backslash G)$ and let $\tilde \phi$ be the lift of $\phi$ to $G$. For all $u \in G$ and all bounded Borel sets $E^+ \subset H^+(u)$,
 \[ \lim_{t\rightarrow +\infty} \int_{E^+}\tilde \phi(ga_t) d \sigma_{H^+(u)} (g) = \sigma_{H^+(u)}(E^+)\frac{m^{\operatorname{BMS}}(\tilde \phi)}{|m^{\operatorname{BMS}}|}.\] 
 \end{theorem}
 \begin{proof} The proof of \cite[Theorem 3.2]{roblin_ergodicite_2003} relies only on the mixing property of geodesic flow, expansion (contraction) properties of the unstable (stable) horospheres, and the local product structure of the BMS measure on $\mathrm{T}^1\tilde X$ with respect to $N^- \times A \times N^+$. Replacing the mixing of geodesic flow with the mixing of frame flow, and the $N^- \times A \times N^+$ product structure for $\mathrm{T}^1\tilde X$ with the $N^-\times A \times M \times N^+$ product structure for $G$, one can adapt Roblin's proof without difficulty. \end{proof}
 \begin{theorem} \label{leavesequidistribute} Let $\phi$ be a bounded $L^1$ function on $\Gamma \backslash G$, and let $\tilde \phi$ be the lift of $\phi$ to $G$. For $r > 0$ let $M_r(\phi)$ be the measurable function defined $\tilde m^{\operatorname{BMS}}$ almost everywhere by
\[ M_r(\phi)(g) =\frac{1}{\sigma_{H^+(g)} (B^+(g, r))}\int_{B^+(g, r)}\tilde \phi\thinspace  d\sigma_{H^+(g)}.  \]
Then,
\[ M_r(\tilde \phi) \rightarrow \frac{1}{|m^{\operatorname{BMS}}|} \int_{\Gamma\backslash G} \phi\thinspace dm^{\operatorname{BMS}} \]
 in $L^1(\Gamma \backslash G, m^{\operatorname{BMS}})$. 
\end{theorem}
\begin{proof} The proof of \cite[Theorem 6.1]{roblin_ergodicite_2003} applies with only very minor modifications. The case of $\phi$ compactly supported and smooth follows immediately from Theorem \ref{R3.2} and $A$ invariance of the BMS measure. The general case then follows from a density argument using a self adjointness lemma \cite[Chapter 6]{roblin_ergodicite_2003}.  \end{proof}
\begin{theorem}  Suppose that $\nu^-, \nu^+$ are non-atomic $\Gamma$ invariant conformal densities of dimensions $\delta_-, \delta_+$. Let $\delta$ be the critical exponent of $\Gamma$. For all pairs $\psi_i \in C_c(\Gamma \backslash G)$, one has
\[ \lim_{t \rightarrow \infty} e^{(\delta_+ - \delta)t}\int_{\Gamma \backslash G } \psi_1(g) \psi_2(ga_t) dm^{\nu^-, \nu^+}(g)  = \frac{m^{\nu^-,\sigma }(\psi_1)  m^{\sigma, \nu^+}(\psi_2) }{|m^{\operatorname{BMS}}|}.\]
\end{theorem}
\begin{remark} Theorem \ref{T1.3} is obtained as the special case $\nu^- = \nu^+ = \lambda$ (see the notation of Paragraph \ref{ss5.3}) \end{remark}
\begin{proof} The proof of \cite[Theorem 3.4]{roblin_ergodicite_2003} applies with only minor modifications. The key points are the local product structure of BMS measure, the equidistribution of horospheres given by Theorem \ref{R3.2}, and the expansion and contraction properties of $N^+$ and $N^-$. \end{proof}

\subsection{Measure classification for the horospherical groups.}  \label{ss7.3} As a consequence of Theorem \ref{leavesequidistribute} we obtain ergodicity of $m^{\operatorname{BR}}$ under $N^+$.
\begin{corollary} \label{c6.4} The Burger-Roblin measure $m^{\operatorname{BR}}$ is $N^+$ ergodic. 
\end{corollary}

\begin{proof}
Choose an $N^+$ invariant set $B\subset \Gamma \backslash G$ with $m^{\operatorname{BR}}(B) > 0$. Without loss of generality we may assume that $B$ is strictly $N^+$ invariant, i.e. that $B = BN^+$. Then $m^{\operatorname{BMS}}(B)>0$ as well. Let $\chi_B$ be the indicator function of $B$. We consider the function $M_r(\tilde \chi_B)(g)$  from Theorem \ref{leavesequidistribute}. For each $g$ this is either identically zero or identically one (independent of $r$). Thus $ \frac{m^{\operatorname{BMS}}(\chi_B)}{|m^{\operatorname{BMS}}|} $ is either zero or one. It can't be zero so it must be one. It follows that $m^{\operatorname{BMS}}(B^c)=0$. But then $m^{\operatorname{BR}}(B^c)=0$ as required. 
\end{proof}

\begin{corollary} \label{c6.5}
Let $\nu$ be an $N^+\negthickspace$ invariant Radon measure on $\Gamma \backslash G$ such that the integral over all right $M$ translates of $\nu$ is equal to $m^{\operatorname{BR}}$, i.e.
\[ \int_M m_* \nu(\psi) dm = m^{\operatorname{BR}}(\psi) \]
for all $\psi \in C_c(\Gamma \backslash G)$. Then $\nu = m^{\operatorname{BR}}$. 
\end{corollary}

\begin{proof}
Since $m^{\operatorname{BR}}$ is ergodic, and since the measures $m_* \nu$ are right $N^+$ invariant Radon measures we have that $m_*\nu = m^{\operatorname{BR}}$ almost surely. The result follows. 
\end{proof}

In fact $m^{\operatorname{BR}}$ is not the only possible $N^+$ invariant and ergodic measure on $\Gamma \backslash G$: there are others coming from closed $MN^+$ orbits, which we aim to understand next.  Suppose that $g \in G$ and that $\Gamma gMN^+ \subset G$ is closed; this occurs either when $g^-$ is outside the limit set, or when it is a parabolic fixed point of $\Gamma$. We would like to understand the orbit closure $\overline{\Gamma gN^+}$. Let $M_0(g)$ be the virtually abelian group (see \cite[Theorem 8.24]{raghunathan_discrete_1972})
\[ M_0(g) = \lbrace m \in M \mbox{ such that } mh \in g^{-1}\Gamma g \mbox{ for some } h \in N^+ \rbrace. \]

\begin{remark}
It is useful to note that 
\[ \mbox{stab}_\Gamma(g^-) = \mbox{stab}_\Gamma(gMN^+) =  \mbox{stab}_\Gamma(g\overline{M_0(g)}N^+ ) = \mbox{stab}_\Gamma(gM_0(g)N^+ ) . \]
\end{remark}

\begin{lemma} Suppose that $\Gamma$ is geometrically finite. If $g^-$ is a parabolic fixed point of $\Gamma$, or is outside the limit set $\Lambda(\Gamma)$, then the collection of closed sets
\[ \lbrace \gamma g\overline{M_0(g)}N^+: \gamma \in \Gamma / \mathrm{stab}_\Gamma(g^-) \rbrace \subset G \]
is locally finite. In particular, the union is closed, and 
\[ \overline{ \Gamma g N^+} = \Gamma g N^+ \overline{M_0(g)}.\] 
\end{lemma}

 \begin{proof}
We choose a disjoint and $\Gamma$ equivariant collection of  open horoballs $\lbrace B_\xi: \xi \in \Gamma g^- \rbrace $ based at the translates of $g^-$. We may further assume that the closures $\overline {B_\xi}$ are also pairwise disjoint. 

Now consider $h \in G$. If $h^- \notin \Gamma g^-$, then $h$ belongs to the complement $  G \setminus  \Gamma gMN^+$, which gives an open neighborhood of $h$ disjoint from the family $\lbrace \gamma g\overline{M_0(g)}N^+: \gamma \in \Gamma \rbrace  $.

 Conversely, if $h^- = \gamma_0 g^-$, then the projection to $\tilde X$ given by  $\pi(ha_t)$ belongs to  $B_{\gamma_0 g^-}$ for some $t$. Decreasing $t$ if necessary, we may also assume that $\pi(ga_tMN^+) \subset B_{g^-}$. But then $\pi^{-1} (B_{\gamma_0 g^-})a_{-t}$ is an open neighborhood of $h$ in $G$, and
\[ \pi^{-1}( B_{\gamma_0 g^-}) a_{-t} \cap \gamma  g\overline{M_0(g)}N^+ \neq \emptyset \]
only when $\gamma \sim \gamma_0$ in $\Gamma /  \mbox{stab}_\Gamma(g^-)$. It follows immediately that $\Gamma g \overline{M_0(g)}N^+$ is closed. On the other hand,
\[ \Gamma gN^+ = \Gamma g M_0(g) N^+, \]
so
\[ \overline {\Gamma gN^+} = \Gamma g\overline{ M_0(g)} N^+.\]
 \end{proof}
With this in mind it is natural to define the measure $\rho_g$ as the unique (up to proportionality) $\overline{ M_0(g)} N^+$ invariant measure supported on $\Gamma \backslash \overline {\Gamma gN^+} $. This measure is right invariant by $N^+$ and locally finite.
\begin{lemma} Suppose that $\Gamma < G$ is geometrically finite, and let $g \in G$ such that $g^-$ is not a radial limit point of $\Gamma$. The measure $\rho_g$ is ergodic for the $N^+$ action on $\Gamma \backslash G$.
\end{lemma}

\begin{proof} Denote by $\tilde \rho_g$ the lift of $\rho_g$ to $G$. 
Suppose that $B \subset G$ has positive measure for $\tilde \rho_g$, and is left $\Gamma$ invariant and right $N^+$ invariant. Without loss of generality, assume that $B \subset \Gamma g\overline{M_0(g)}N^+$. Then $ B\cap g\overline{M_0(g)}N^+$ also has positive measure for $\tilde \rho_g$. We observe that $B$ is invariant for the action of $M_0(g) \times N^+$ on ${g\overline{M_0(g)}N ^+\sim  \overline{M_0(g)} N^+}$ by left translation. But dense subgroups are ergodic for the Haar measure, so 
\[ B\cap g\overline{M_0(g)}N^+ \subset g\overline{M_0(g)}N^+ \]
has full $\tilde \rho_g$ measure and $B \subset G$ has full $\tilde \rho_g$ measure. 

\end{proof}

\begin{lemma} \label{l7.8} Suppose that $\Gamma < G$ is geometrically finite and that $g \in G$ such that $g^-$ is not a radial limit point of $\Gamma$. Any measure on $\Gamma \backslash G$ that is invariant and ergodic for the right $N^+$ action and is supported on a single closed $MN^+$ orbit $\Gamma \backslash \Gamma gMN^+$ is actually supported on a single orbit $\Gamma \backslash \Gamma g'\overline{M_0(g')}N^+$ for some $g' \in gM$. 
\end{lemma}

\begin{proof} Let $\nu$ be such a measure and pick $g'$ in the support of $\nu$. For $\epsilon > 0$ let $M_\epsilon$ be an $\epsilon$ neighborhood of the identity in $M$. The set $\Gamma \backslash \Gamma g' \overline{M_0(g')}M_\epsilon N^+ $ is  $N^+$ invariant, and has positive $\nu$-measure, so has full $\nu$-measure. It follows that $\Gamma \backslash \Gamma g' \overline{M_0(g')} N^+ $ has full $\nu$-measure as expected. \end{proof}

\begin{lemma} \label{l7.9}
 Suppose that $\Gamma < G$ is geometrically finite, and that $g\in G$ such that $g^-$ is not a radial limit point of $\Gamma$. Any Radon measure $\nu$ on $\Gamma \backslash G$ that is invariant for the right $N^+$ action, is supported on $\Gamma \backslash \Gamma g\overline{M_0(g)}N^+$, and satisfies 
\[ \int_M m_* \nu = \int_M m_*\rho_g \]
is equal to $\rho_g$. 
\end{lemma}

\begin{proof}
Since 
\[ \int_M m_* \nu = \int_M m_*\rho_g \]
and since both $\nu$ and $\rho_g$ are supported on $\Gamma \backslash \Gamma g\overline{M_0(g)}N^+$ we see that
\[ \int_{\overline{M_0(g)}} m_* \nu = \int_{\overline{M_0(g)}} m_* \rho_g=  \rho_g. \]
Since $\rho_g$ is  ergodic it follows that $m_* \nu = \rho_g$ for almost every $m \in \overline{M_0(g)}$. Thus $\nu =\rho_g$ as expected. 
\end{proof}

\begin{theorem} \label{classifyfNmeasures} Suppose that $\Gamma < G$ is geometrically finite. Up to proportionality, the only measures on $\Gamma \backslash G$ that are invariant and ergodic for the right $N^+$ action are 
\begin{itemize}
\item{ the $M$ invariant lift $m^{\operatorname{BR}}\negthickspace$ of the Burger-Roblin measure, and }
\item{ measures of the form $\rho_g$ where $g^-$ is not in the radial limit set.}
\end{itemize}
\end{theorem}

\begin{proof}
Let $\nu$ be such a measure and consider the $MN^+$ invariant and ergodic measure $ \hat \nu =  \int_M m_* \nu $ obtained by averaging over $M$ orbits. By Roblin (see \cite{roblin_ergodicite_2003}) we have either that $\hat \nu $ is proportional to $m^{\operatorname{BR}}$, in which case $\nu$ is proportional to $m^{\operatorname{BR}}$ (Corollary \ref{c6.5}), or that $\hat \nu$ is the homogeneous $MN^+$ invariant measure on some closed $MN^+$ orbit $\Gamma \backslash \Gamma gMN^+$. In this second case Lemma \ref{l7.8} tells us that $\nu$ is supported on some closed orbit $ \Gamma \backslash \Gamma g' \overline{M_0(g')} N^+$, and (see \cite{roblin_ergodicite_2003} again) that
\[ \hat \nu = \int_M m_* \nu = \int_M m_*\rho_{g'}\]
up to proportionality. Then Lemma \ref{l7.9} says that $\nu = \rho_{g'}$ up to proportionality.  \end{proof}

\subsection{Non-dense transitivity groups.} \label{ss7.4}
It is natural to ask what happens if the transitivity groups are not dense in $MA$. We're now in a position to provide a partial answer. We will construct a closed $\Gamma$ and $N^+$ invariant subset of $G$. For $g$ in the support of the BMS measure, we define a subset $\mathcal{D}(g) \subset \operatorname{supp}(m^{\operatorname{BMS}})$ by
\[ \mathcal{D}(g) = \lbrace g' \in \mathrm{supp}(m^{\operatorname{BMS}}): g' \mbox{ is reachable from } g \rbrace; \]
see the definition of reachability at the beginning of Section \ref{sec3}. This is, of course, strongly related to the definition of the transitivity group. One consequence of that relationship is the following lemma.
\begin{lemma}
For any $g \in \mathrm{supp}(m^{\operatorname{BMS}})$ we have
\[ \mathcal{D}(g) \cap gMAN^-N^+ =g \mathcal{H}_\Gamma(g) N^-N^+ \cap \operatorname{supp}(m^{\operatorname{BMS}}).\]
\end{lemma}
\begin{proof}
It is clear that 
\[ g \mathcal{H}_\Gamma(g) N^-N^+ \cap \operatorname{supp}(m^{\operatorname{BMS}}) \subset \mathcal{D}(g) \cap gMAN^-N^+.\]
On the other hand, suppose that $g' \in \mathcal{D}(g) \cap gMAN^-N^+$. Choose  ${a \in A},\\ {m\in M}, n \in N^-$ and $h \in N^+$ such that $g' = gmanh$. Also choose $\gamma \in \Gamma$ and a sequence $h_i \in N^-\cup N^+$ such that  
\[ g' = \gamma g h_1 \ldots h_k \mbox{ and } gh_1 \ldots h_r \in  \operatorname{supp}(m^{\operatorname{BMS}})\]
for all $1 \leq r \leq k$. Then $h_1, \ldots, h_k,  h^{-1}, n^{-1}$ and $\gamma$ is a sequence for $ma $ in $\mathcal{H}_\Gamma(g)$ and $g' \in g\mathcal{H}_\Gamma(g)N^-N^+ \cap  \operatorname{supp}(m^{\operatorname{BMS}})$ as required. 
\end{proof}

\begin{remark}
The relation $g_1 \sim g_2$ if and only if $g_1 \in \mathcal{D}(g_2)$ defines an equivalence on $\operatorname{supp}(m^{\operatorname{BMS}})$. 
\end{remark}

\begin{remark}
Both the equivalence classes $\mathcal{D}(g)$ and their closures $\overline{\mathcal{D}(g)}$ are $\Gamma$ invariant. 
\end{remark}
\begin{corollary}  \label{ifnotdense} For each $g \in \mathrm{supp}(m^{\operatorname{BMS}})$ the set $\overline{ \mathcal{D}(g)}N^+$ is closed, right $N^+$ invariant and left $\Gamma$ invariant. In addition we have
\[ \overline{ \mathcal{D}(g)}N^+ \cap gN^-MAN^+ = (gN^- \cap \operatorname{supp}(m^{\operatorname{BMS}}) )\overline{\mathcal{H}_\Gamma(g)} N^+.\]
In particular, if $\mathcal{H}_\Gamma(g)$ is not dense in $MA$, then $m^{\operatorname{BR}}$ is not $N^+$ ergodic on $\Gamma \backslash G$,  and $m^{\operatorname{BMS}}$ is not mixing for frame flow on $\Gamma \backslash G$.
\end{corollary}
\begin{proof}
The statements on closure and $\Gamma$ and $N^+$ invariance of $\overline{ \mathcal{D}(g)}N^+$ are immediate. If we assume now that the transitivity group $\mathcal{H}_\Gamma(g)$ is not dense in MA then we may choose a small neighborhood $M_\epsilon A_\epsilon$ of the origin in $MA$ such that $ \overline{\mathcal{H}_\Gamma(g)} M_\epsilon A_\epsilon$ has positive measure (for the Haar measure on $MA$) but not full measure. The set $\Gamma \backslash \overline{ \mathcal{D}(g)}N^+M_\epsilon A_\epsilon$ is $N^+$ invariant and has positive BR measure, but not full BR measure, and so we conclude that $m^{\operatorname{BR}}$ is not $N^+$ ergodic. Non-mixing of frame flow follows by Corollary \ref{c6.4}. 
\end{proof}
Even in the case that the transitivity group is not dense, however, one can sometimes still salvage a mixing result.

\begin{theorem}
Suppose that $|m^{\operatorname{BMS}}| < \infty$ and fix $g\in\mathrm{supp}(m^{\operatorname{BMS}})$. The set $\Gamma \backslash  \overline{\mathcal{D}(g)A}$ is $\overline {\mathcal{H}^w_\Gamma(g)}A$ invariant, and carries a unique $\overline{\mathcal{H}^w_\Gamma(g)}A$ invariant measure $\theta_g$ whose projection to $\Gamma \backslash G / M$ coincides with $m^{\operatorname{BMS}}$. The measure $\theta_g$ is $A$ ergodic, and the $A$ action on $(\Gamma \backslash G, \theta_g)$ is mixing whenever $\overline{\mathcal{H}_\Gamma(g)} = \overline{\mathcal{H}_\Gamma ^w (g)} A < M A$. One has $m_* \theta_g = \theta_{gm}$ for all $m \in M$, and any two such measures $\theta_g, \theta_{g'}$ with $g, g'\in \mathrm{supp}(m^{\operatorname{BMS}})$ are related by some element $m \in M$. 
\end{theorem}

\begin{proof} The existence and $M$ equivariance properties of $\theta_g$ are clear from the local product structure of $G$, while uniqueness follows by an argument similar to that for Lemma \ref{l7.9}. It remains to check that the ergodicity and mixing properties hold. 

First we aim to prove $A$ ergodicity. Consider a function $\psi \in C_c (\Gamma \backslash G)$. Let 
\[ \psi_T(h) := \frac{1}{T} \int_0^T \psi(ha_t) dt\]
be the ergodic average of $\psi$. Write $\hat \Sigma$ for the sigma algebra of all $A$ invariant Borel subsets of $\Gamma \backslash G$, and $\psi_\infty := \mathbb{E}(\psi| \hat\Sigma, m_g)$ for the expectation of $\psi$ conditioned on $\hat \Sigma$. We aim first to show that $\psi_\infty$ is $\theta_g$ almost everywhere constant. The Birkhoff-Khinchin Theorem tells us that $\psi_T\rightarrow \psi_\infty$ almost everywhere as $T\rightarrow \pm \infty$. It is well known that point wise limits of $\psi_T$ as $T\rightarrow +\infty$ (as $T\rightarrow - \infty$ respectively) are $N^-$ (respectively $N^+$) invariant. The arguments of Section \ref{sec5} now apply with minimal modification to tell us that $\psi_\infty$ is invariant for the right $\overline {\mathcal{H}^w_\Gamma(g)}A$ action. Changing the function on a set of $\theta_g$ measure zero we may then assume that $\psi_\infty$ is $MA$ invariant. Hence $\psi_\infty$ is $m^{\operatorname{BMS}}$ almost everywhere constant; this together with strict $M$ invariance implies that $\psi_\infty$ is $\theta_g$ almost everywhere constant.

The map $\psi \mapsto \psi_\infty$ is orthogonal projection to the space $L^2(\Gamma \backslash G, \hat \Sigma, \theta_g)$ of $\hat \Sigma$ measurable functions; it is continuous and surjective. Since every continuous function is projected to a constant function we conclude that the only $\hat \Sigma$ measurable functions are constants, and so that $\theta_g$ is $A$ ergodic. 

The proof of the mixing property when $\overline{\mathcal{H}_\Gamma(g)} = \overline{\mathcal{H}_\Gamma ^w (g)} A$ is similar to the Zariski dense case. We fix $\psi \in C_0(\Gamma \backslash G)$ and aim to show that any $L^2(\theta_g)$ weak limit of $\psi\circ a_t$ is constant. The arguments of Section \ref{sec5} imply that $\psi$ is invariant for the action of $\overline{\mathcal{H}_\Gamma (g)} = \overline {\mathcal{H}^w_\Gamma(g)}A$. Ergodicity of the $A$ action on $(\Gamma \backslash G, \theta_g)$ then completes the argument. 
\end{proof}

\begin{remark}
Suppose that $G= \operatorname{PSL}_2(\mathbb{C})$, that $\Gamma$ is non-elementary, and that $|m^{\operatorname{BMS}}| < \infty$. Then Lemma \ref{H^3 case} tells us that the measures $\theta_g$ (where $g \in \mathrm{supp}(m^{\operatorname{BMS}})$) are always mixing. 
\end{remark}

\section{Lifting the Bernoulli property} \label{sec8}

We will finish by showing that the frame flow has stronger properties than being mixing. Throughout this Section we assume the conditions of Theorem \ref{target}; namely that $\Gamma <G$ is Zariski dense and that $m^{\operatorname{BMS}}$ is finite. See \cite{cornfeld1982ergodic} for the relevant definitions from ergodic theory (in particular for the notions of K-system and Bernoulliness). 

\begin{theorem} \label{K}
The $A$ action on $(\Gamma \backslash G, m^{\operatorname{BMS}})$ defines a $K$-system. 
\end{theorem}

\begin{corollary} \label{Bernoulliness}
Suppose further that the $A$ action on $(\Gamma \backslash G/M, m^{\operatorname{BMS}})$ is Bernoulli. Then the $A$ action on $(\Gamma \backslash G, m^{\operatorname{BMS}})$ is also Bernoulli. 
\end{corollary}

\begin{proof}[ Proof of Corollary \ref{Bernoulliness}] Rudolph \cite{rudolph_classifying_1978} showed that compact group extensions of Bernoulli flows are Bernoulli if and only if they are mixing. The corollary is then immediate from Theorem \ref{target}. 
\end{proof}

\begin{lemma} \label{l7.3}
Any set $B \subset \Gamma \backslash G$ measurable with respect to the Pinsker algebra $\pi(a_1)$ is almost invariant for the $N^-$ action;  there is a measurable set $B_-$ with $m^{\operatorname{BMS}}(B \triangle B_-) = 0$ and $B_- = B_-N^-$.
\end{lemma}

\begin{proof}
Write $T$ for the automorphism of $\Gamma \backslash G$ induced by $a_1$. By \cite[Section 9]{margulis_invariant_1994} we may choose a measurable partition $\mathcal{P}$ of $\Gamma\backslash G$, a positive number $\epsilon$, and a measurable set $F \subset \Gamma \backslash G$ of positive BMS measure such that 

\begin{itemize}
\item{ $T\mathcal{P}$ is a refinement of $\mathcal{P}$, and  }
\item{ for all $ x \in F$ the $\mathcal{P}$-atom of $x$ satisfies  $B^-(x, \epsilon) \subset [x]_\mathcal{P} \subset B^-(x, \epsilon^{-1})$ (see subsection \ref{ss7.1} for notation).}
\end{itemize}
Define the set
\[ E  = \lbrace x \in \Gamma \backslash G : T^n(x) \in F \mbox{ infinitely often for $n$ positive and for $n$ negative}  \rbrace. \]
This set is $T$ invariant and has full measure. We write $\tilde {\mathcal{P}}$ and $\tilde T$ for the restrictions of $\mathcal{P}$ and $T$ to $E$. Also write $\mathcal{Q}$ for the partition of $\Gamma \backslash G$ into $N^-$ leaves, and $\tilde{\mathcal{Q}}$ for the restriction of $\mathcal{Q}$ to $E$. It is not hard to see that 

\begin{itemize}
\item{ $\tilde T\tilde{ \mathcal{P}}$ is a refinement of $\tilde{\mathcal{P}}$.  }
\item{ For every $x \in E$ the atom for the product of the partitions $\tilde T^n \tilde{ \mathcal{P }}$ is $[x]_{\Pi^{n} \tilde T^n \tilde{ \mathcal{P}}} = \lbrace x \rbrace$. }
\item{The intersection of the partitions $\tilde T^n \tilde{ \mathcal{P }}$ is $\tilde{\mathcal{Q}}$.}
\end{itemize}
Theorem 2 of \cite{rohlin_construction_1961} tells us  that $\pi(\tilde T) \leq \tilde{\mathcal{Q}}$. Now if $B\subset \Gamma \backslash G$ is measurable by the Pinsker algebra $\pi(T)$, we know that the associated partition $\lbrace B, B^c \rbrace$ has zero entropy. It follows the partition of $E$ generated by $\tilde B := B \cap E$ has zero entropy, so is measurable with respect to $\pi(\tilde T)$. Thus $\tilde B$ is measurable with respect to $\tilde{\mathcal{Q}}$.  It is then an  exercise to modify $B$ by a set of measure zero to make it strictly $N^-$ invariant. 

\end{proof}

\begin{lemma} \label{l7.4}
Any set $B \subset \Gamma \backslash G$ measurable with respect to the Pinsker algebra $\pi(a_1)$ is almost invariant for the $N^+$ action;  there is a measurable set $B_+$ with $m^{\operatorname{BMS}}(B \triangle B_+) = 0$ and $B_+ = B_+N^+$.
\end{lemma}
\begin{proof}
Since the partition $\lbrace B, B^c \rbrace$ has zero entropy for $a_1$ it also has zero entropy for $a_{-1}$. The proof of the preceding lemma now applies. 
\end{proof}

\begin{proof}[ Proof of Theorem \ref{K}]  Let $\pi(a_1)$ be the Pinsker $\sigma$-algebra of the time one frame flow on $\Gamma \backslash G$. Lemmas  \ref{l7.3} and \ref{l7.4} imply that $\pi(a_1) < \tilde \Sigma$. But $\tilde \Sigma$ is trivial in our setting by Corollary \ref{C5.3}. Thus $a_1$ has totally positive entropy. It follows (see \cite{sinai_dynamical_1964}) that $a_1$ is a K-automorphsim and so (see \cite[Section 8, Theorem 2]{cornfeld1982ergodic}) that $a_t$ is a K-system. 
\end{proof}

\bibliography{working_copy.bib}{}
\bibliographystyle{alpha}

\end{document}